\documentclass[12pt]{amsart}
\pdfoutput=1
\usepackage{amssymb,latexsym}
\usepackage{amsfonts}
\usepackage{amsmath}
\usepackage{enumitem}
\usepackage{marvosym}
\usepackage{lipsum}
\usepackage{float}
\usepackage{tikz-cd}

\usepackage{amsthm}
\usepackage[square, comma, sort&compress, numbers]{natbib} 
\usepackage{fancyhdr} 
\usepackage{lastpage}
\usepackage{layout} 
\usepackage{esint} 
\usepackage{extarrows} 
\usepackage[marginal]{footmisc} 
\usepackage{mathrsfs} 

\usepackage[colorlinks=true, linkcolor=red, citecolor=blue, urlcolor=blue]{hyperref}
\usepackage[a4paper, text={160mm,230mm}, centering, footskip=20pt]{geometry}

\pretocmd{\tableofcontents}{\hypersetup{linkcolor=blue}}{}{}
\apptocmd{\tableofcontents}{\hypersetup{linkcolor=red}}{}{}
\newtheorem{theorem}{Theorem}[section]
\newtheorem{corollary}[theorem]{Corollary}
\newtheorem{lemma}[theorem]{Lemma}
\newtheorem{proposition}[theorem]{Proposition}
\newtheorem{definition}[theorem]{Definition}

\newtheorem{conjecture}[theorem]{Conjecture}

\newtheorem{remark}[theorem]{\bf{Remark}}

\numberwithin{equation}{section}

\newcommand{\bp}{\bar{\partial}}
\newcommand{\beq}{\begin{equation}}
	\newcommand{\eeq}{\end{equation}}
\newcommand{\beqn}{\begin{equation*}}
	\newcommand{\eeqn}{\end{equation*}}

\newcommand{\C}{\mathbb{C}}
\newcommand{\R}{\mathbb{R}}
\newcommand{\N}{\mathbb{N}}

\newcommand{\pp}{\partial\bar{\partial}}
\newcommand{\p}{\partial}

\newcommand{\w}{\omega}

\newcommand{\im}{\sqrt{-1}}

\newcommand{\Aut}{\mathrm{Aut}}
\newcommand{\rd}{\mathrm{rd}}
\newcommand{\loc}{\mathrm{loc}}

\newcommand{\mF}{\mathcal{F}}
\newcommand{\mO}{\mathcal{O}}

\newcommand{\mC}{\mathcal{C}}

\DeclareMathOperator \rank{rank}
\DeclareMathOperator \tr{tr}

\DeclareMathOperator{\sca}{S}

\DeclareMathOperator{\rc}{Ric}
\DeclareMathOperator{\spaned}{span}
\DeclareMathOperator{\hsc}{H}

\title[Partially semi-positive curvature]{Compact K\"ahler manifolds with partially semi-positive curvature}

\author{Shiyu Zhang}
\address{School of Mathematical Sciences, University of Science and Technology of China,
	Hefei, 230026, P.R. China}
\email{{\tt shiyu123@mail.ustc.edu.cn}}

\author{Xi Zhang}
\address{School of Mathematics and Statistics,
	Nanjing University of Science and Technology,
	Nanjing, 210094, P.R.China}
\email{{\tt mathzx@njust.edu.cn}}

\date{\today}


\begin{document}
\begin{abstract}
	In this paper, we study MRC fibrations of compact K\"ahler manifolds with partially semi-positive curvature. We first prove that a compact K\"ahler manifold is rationally connected if the tangent bundle is BC-$p$ positive for all $1\leq p\leq \dim X$.  As applications, we confirm a conjecture asserting that any compact K\"ahler manifold with positive orthogonal Ricci curvature must be rationally connected, and generalize a result of Heier-Wong and Yang to the conformally K\"ahler case. The second main result concerns structure theorems for two intermediate curvature conditions. We prove that, a compact K\"ahler manifold with $k$-semi-positive Ricci curvature or semi-positive $k$-scalar curvature, either the rational dimension $\geq n-k+1$ or it admits a locally constant fibration $f:X\rightarrow Y$ such that the fibre is rationally connected and the image $Y$ is Ricci-flat.
\end{abstract}

\maketitle
\tableofcontents

	\section{Introduction and statements}
	This paper is delicated to study the structure of a compact K\"ahler manifold with certain partially semi-positive curvature condition through the MRC fibration. The terminology ``partially semi-positive curvature conditions'' refer to semi-positive differential geometric curvature conditions that may allow negative curvature in some directions. For clarifying our arguments, the introduction will be divided into parts.
	\subsection{BC positivity and rational connectedness}
     A compact complex manifold $X$ is said to be {\em rationally connected} if any two points of $X$ can be joined by a chain of rational curves. First introduced in the foundational work of \cite{KMM}, such manifolds play a central role in birational geometry (see \cite{De2001} for a detailed account). A recent work of Li-Zhang-Zhang \cite{LZZ21} established that a compact K\"ahler manifold is rationally connected if and only if its tangent bundle is mean curvature positive.
    
     \vspace{0.2cm}
    
     The first objective of this paper is to establish a new differential-geometric criterion for rational connectedness using {\em BC-$p$ positivity}, which is introduced by L. Ni in \cite{ni2021}. We note that in the case $p=1$, this coincides with the RC-positivity defined earlier in \cite{yang2018}.
	\begin{definition}\label{defn-BC-positivity}
		Let $E$ be a holomorphic vector bundle over a complex manifold $X$. We say that $E$ is BC $p$-positive if $E$ admits a Hermitian metric $h$ such that for any $x\in X$ and any $p$-dimensinal subspace $\Sigma\subset E_x$, there exists $v\in T_{X,x}$ such that
		$$\sum\limits_{i=1}^pR(v,\bar{v},e_i,\bar{e_i})=p\fint_{u\in\Sigma,|u|_h=1}R(v,\bar{v},u,\bar{u})d\mu(u)>0,$$
		where $R$ is the Chern curvature tensor of $(E,h)$ and $\{e_i\}_{i=1}^p$ is a unitary basis of $\Sigma$.
	\end{definition}
	
    \noindent Regarding the tangent bundle, BC-$p$ positivity is a very weak curvature condition that can be guaranteed by a wide range of curvature positivity conditions in differential geometry, see \cite[pp. 281, 291–292]{ni2021}, as well as our discussions in Section \ref{BC-application}. By a clever consideration to Whitney's comass, L. Ni proved that BC-$p$ positivity of the tangent bundle implies $h^{p,0}$=0 in \cite{ni2021}. In particular, BC-$2$ positivity ensures the projectivity by Kodaira's embedding Theorem \cite{Kod}. Besides this, the possible consequences of geometric significance remain largely unexplored.
    
    \vspace{0.2cm}
    
    The first main result of this paper is
    \begin{theorem}\label{main-thm-BC-p-1}
    	Let $X$ be a compact complex manifold with BC-$p$ positive tangent bundle for some $p\geq1$. Suppose that $f:X\dashrightarrow Y$ is a dominant meromorphic map between compact complex manifolds such that $\dim Y= p\leq \dim X$, then $K_Y$ is not pseudo-effective.
    \end{theorem}
	\noindent The basic idea is proving that $\log$ of the Hermitian ratio $\psi$--a naturally arising quantity associated with $f^*K_Y$--is strictly subharmonic along some direction at the maximum point of $\psi$ if $K_Y$ is pseudo-effective (Proposition \ref{prop-subharmonic}). The key innovation is that BC-$p$ positivity emerges naturally in a Bochner-type formula (Proposition \ref{prop-Bochner}) as the weakest possible differential-geometric curvature condition of the tangent bundle under which Proposition \ref{prop-subharmonic} remains valid. In the proof, the geometric properties of the meromorphic map takes a similar role as the comass.
	 
	 \vspace{0.1cm}
	 
	 As a direct consequence of Theorem \ref{main-thm-BC-p-1}, Ou's criterion for uniruledness \cite{Ou25} and the existence of the maximal rationally connected fibration (MRC fibration) by Campana \cite{Cam04}, we obtain
	
	\begin{theorem}\label{main-thm-BC-p-2}
		Let $X$ be a compact K\"ahler manifold of dimension $n$. Suppose that for some $1\leq k\leq n$, the tangent bundle is BC-$p$ positive for all $p$ with $k\leq p\leq \dim X$. Then the rational dimension $\rd(X)\geq n-k+1$. In particular, $X$ is rationally connected when $k=1$.
	\end{theorem}
	\noindent The rational dimension $\rd(X)$ is defined as the dimension of the generic fibre of some MRC fibration of $X$, which is a bimeromorphic invariant of $X$. Since positivity of the mean curvature implies BC-$p$ positivity for all $p\geq1$ (Proposition \ref{prop-relationship}), combining Li-Zhang-Zhang's work \cite{LZZ21}, we obtain a new differential-geometric criterion for rational connectedness
	
	\begin{theorem}\label{main-thm-BC-RC}
		A compact K\"ahler manifold is rationally connected if and only if $T_X$ is BC-$p$ positive for every $1\leq p\leq n$.
	\end{theorem}
    \noindent ``$T_X$ is BC-$p$ positive for $1\leq p\leq n$'' can be directly implied by the positivity of several curvature conditions, including well-known holomorphic sectional curvature (more generally, $k$-Ricci curvature in \cite{nicpam}), uniformly RC-positivity in \cite{yang2019}, mean curvature, etc. While rational connectedness under these curvature conditions was previously established in \cite{Cam92,KMM,ni2021,yang2018,yang2019}, Theorem \ref{main-thm-BC-RC} provides a novel and unified approach to proving this property. In the following, we will demonstrate that this criterion provide new perspectives in the broader settings.
    
    \vspace{0.3cm}
    
    For any $v\in T_{X,x}$ with $|v|=1$, the orthogonal Ricci curvature $\rc^\perp(v,\bar{v})$ is defined as $$\rc^\perp(v,\bar{v}):=\rc(v,\bar{v})-\frac{R(v,\bar{v},v,\bar{v})}{|v|^2}=\sum\limits_{i=1}^{n-1}R(v,\bar{v},e_i,\overline{e_i}),$$
    where $\{v,e_1,\cdots,e_{n-1}\}$ is a unitary frame of $T_{X,x}$. This curvature positivity has been well studied in \cite{NWZ21,NZ19,ni2018,ni2021} motivated by studying Comparison theorems \cite{ni2018} and the generalize Hartshorne conjecture (see e.g. \cite{CP91}). As a first application of Theorem \ref{main-thm-BC-RC}, we provide an affirmative answer to a conjecture by Ni-Wang-Zheng \cite{NWZ21}.
	
	\begin{theorem}\label{main-thm-orthogonal-positive}
		A compact K\"ahler manifold with $\rc^\perp>0$ is rationally connected.
	\end{theorem}
    
    \noindent This statement can be extended to the quasi-positive case via a Bochner-type integral inequality \eqref{equa-integral-inequality-1} (see Corollary \ref{coro-orthogonal-quasi-positive}).

    \vspace{0.3cm}
    
    As a second application, we generalize Yau's conjecture on positive holomorphic sectional curvature \cite[Problem 47]{Yau82}--originally proved by Heier-Wong and Yang in \cite{heier2020,yang2018}--to the conformally K\"ahler case.
	\begin{theorem}\label{main-thm-HSC-conformallyKahler}
		A compact conformally K\"ahler manifold with positive holomorphic sectional curvature is rationally connected.
	\end{theorem}

	\noindent The proof is different from \cite{heier2020,yang2018,yang2019} in the K\"ahler case. Recall that Yang's proof relies on establishing that $(\Lambda^pT_X,\Lambda^pg)$ is RC-positive, then follows from his result that RC-positivity of $\Lambda^pT_X$ for every $1\leq p\leq n$ implies the rational connectedness. However, it cannot be verified in our setting because the variational argument employed by Yang fails to apply in this case. Instead, our novel approach involves constructing, for each $1\leq p\leq n$, a Hermitian metric $h_p$ that makes $(T_X,h_p)$ BC-$p$ positive instead of $\Lambda^p T_X$ RC-positive (Proposition \ref{prop-conformallykahler-BC}). Theorem \ref{main-thm-HSC-conformallyKahler} then follows directly from Theorem \ref{main-thm-BC-RC}. This verifies a special case of the following conjecture that the Hermitian metric is conformally K\"ahler and its curvature is strictly positive.
	
	\begin{conjecture}\label{main-conj-HSC-hermitian}(Yang, \cite[Conjecture 1.9]{yang2019})
		A compact K\"ahler manifold $X$ admits a Hermitian metric $h$ with quasi-positive holomorphic sectional curvature is projective and rationally connected.
	\end{conjecture}
	
	From the algebraic perspective of the curvature tensor at one point, BC-$p$ positivity of $(T_X,h)$ is strictly weaker than RC-positivity of $(\Lambda^pT_X,\Lambda^p h)$ (see Lemma \ref{lem-BC-RC}). Actually, Yang proposed a stronger conjecture in \cite{yang2018}.
	\begin{conjecture}[\cite{yang2018}]
		A compact K\"ahler manifold is rationally connected if and only if $T_X$ is BC-1 positive.
	\end{conjecture}
	\noindent By Theorem \ref{main-thm-BC-RC}, this conjecture is equivalent to that BC-1 positivity of the tangent bundle would imply BC-$p$ positivity for all $p\geq2$.
	
	\subsection{Structure theorems for immediate curvature conditions} As a natural extension of the study of compact K\"ahler manifolds with strictly positive curvature, investigations into the structure theorems of those satisfying semi-positive curvature conditions have garnered considerable attention in recent decades. Specifically, the study on semi-positive case of holomorphic bisectional curvature \cite{mok,HSW,SS80}, Ricci curvature \cite{campana2015} (or more generally, nef anti-canonical bundle \cite{CH19,MWWZ25}) and holomorphic sectional curvature \cite{matsumura2022,ZZ25,matsumura2025} have reached a satisfactory level of development. We refer the reader to a survey written \cite{matsumurasurvey} for a more detailed account. A natural question in this direction is whether analogous structure theorems hold under weaker differential-geometric curvature positivity of the tangent bundle, such as $BC$-semi-positivity. It is indeed unclear whether BC-$2$ quasi-positivity implies $h^{2,0}=0$, which was posed by L. Ni in \cite{ni2025}. Motivated by this question and Theorem \ref{main-thm-BC-RC}, it's natural to ask {\em whether a compact K\"ahler manifold is rationally connected if the tangent bundle is BC-$p$ quasi-positive for every $p\geq1$.} These questions remain quite challenging at the moment.
	
	\vspace{0.1cm}
	
	The second main objective of this paper is to investigate the geometric structure of compact K\"ahler manifolds under two natural immediate curvature conditions: $k$-semi-positive Ricci curvature and semi-positive $k$-scalar curvature.
	\begin{definition}(c.f. \cite[Section 2.1]{heier2020})
		Let $X$ be a compact complex manifold of dimension $n$ with a K\"ahler metric $g$.
		\begin{itemize}
			\item[(1)] Ricci curvature is said to be {\em $k$-semi-positive} if at each point $x\in X$ the sum of any $k$ eigenvalues of Ricci curvature is non-negative, which is equivalent to that for any $k$-dimensional subspace $\Sigma\subset T_{X,x}$
			$$\fint_{v\in \mathbb{S}^{2k-1}\subset \Sigma}\rc(v,\bar{v})d\mu(v)=\sum\limits_{i=1}^k\rc(e_i,\overline{e_i})\geq0,$$
			where ${e_i}$ is a unitary basis of $\Sigma$. We denote it by $\rc\geq_k0$.
			\item[(2)] For any $x\in X$ and any $k$-dimensional subspace $\Sigma\subset T_{X,x}$, $k$-scalar curvature along $\Sigma$ is defined by
			$$\sca_k(x,\Sigma):=\frac{k(k+1)}{2}\fint_{v\in \mathbb{S}^{2k-1}\subset \Sigma} R(v,\bar{v},v,\bar{v})d\mu(v)=\sum\limits_{i,j=1}^kR(e_i,\overline{e_i},e_j,\overline{e_j}),$$
			where ${e_i}$ is a unitary basis of $\Sigma$. We say that $k$-scalar curvature is semi-positive if at each point $x\in X$, $S_k(x,\Sigma)\geq0$ for any $k$-dimensional subspace $\Sigma\subset T_{X,x}$.
		\end{itemize}
	\end{definition}

    \noindent $\rc\geq_k0$ interpolates between semi-positive Ricci curvature $\rc\geq0$ and semi-positive scalar curvature $\sca\geq0$. $\sca_k\geq 0$ interpolates between semi-positive holomorphic sectional curvature $\hsc\geq0$ and semi-positive scalar curvature $\sca\geq0$. Clearly, $\rc\geq_k0\Rightarrow\rc\geq_{k+1}0$ for any $k\geq1$. Similarly, $\sca_k\geq0 \Rightarrow\sca_{k+1}\geq0$ (see Lemma \ref{lem-k-k+1}). The classical Bochner formula \cite{Boch1946} establishes that $k$-quasi-positive Ricci curvature implies the vanishing of $h^{p,0}$ for all $p \geq k$. A recent work of Ni and Zheng \cite{ni2022} showed that the same vanishing property holds under the assumption of positive $k$-scalar curvature. Further geometric implications were obtained by Heier and Wong \cite{heier2020}, who proved that for projective manifolds satisfying either Ricci curvature is $k$-quasi-positive or $\sca_k$ is quasi-positive, then its rational dimension $\rd(X)\geq n-k+1$.
    
    \vspace{0.1cm}
    
    The semi-positive case, however, remains unexplored, with the exception of projectivity results under restricted holonomy assumptions established by L. Ni \cite{ni2025} (see also \cite{Tang2024} for the quasi-positive case when $k=2$). We will prove the following structure theorems. For simplicity, we adopt the notations in Definition~\ref{defn-locallyconstant} for a locally constant fibration $f:X\rightarrow Y$.
    \begin{theorem}\label{main-thm-kricci}
    	Let $(X,g)$ be a compact K\"ahler manifold with $\rc\geq_k0$ (resp. $\sca_k\geq0$). Then one of the following situations occurs:
    	\begin{itemize}
    		\item $\rd(X)\geq n-k+1$.
    		\item $\rd(X)\leq n-k$. There exists a locally constant fibration $f:X\rightarrow Y$ to a compact K\"ahler manifold $Y$ such that
    		\begin{itemize}
    			\item[(1)] $Y$ admits a Ricci-flat K\"ahler metric $g_Y$ (resp. a K\"ahler metric $g_Y$ with $\rc\equiv0$ and $\sca_k\equiv0$)
    			\item[(2)] The fibre $F$ is rationally connected and admits a K\"ahler metric $g_F$ with $\rc\geq0$ (resp. a K\"ahler metric $g_F$ with $\hsc\geq0$) such that we have the following holomorphically and isometrically splitting for the universal cover 
    			$$(X_{\mathrm{univ}},\mu^*g)\simeq (Y_{\mathrm{univ}},\pi^*g_Y)\times(F,g_F).$$
    		\end{itemize}
    	\end{itemize}
    \end{theorem}
    
    \noindent Theorem \ref{main-thm-kricci} generalize important earlier work: specifically, they extend the structure theorems of Campana-Demailly-Peternell \cite{campana2015} for manifolds with $\rc\geq0$, and those of Matsumura \cite{matsumura2025} for manifolds with $\hsc\geq0$.
    
    Let us explain the proof of Theorem \ref{main-thm-kricci}. Actually, we prove a structure theorem (Theorem \ref{main-thm-partially}) under the more generalized curvature condition $\sca_{a,b,k}\geq0$. The consideration of this curvature condition is chiefly for technical unification that simultaneously addresses both $\rc\geq_k0$ ($a=1,b=0$), $\sca_k\geq0$ ($a=0,b=1$) and $\rc^\perp\geq0$ ($a=1,b=-1,k=1$).
    \begin{remark}
      A recent work of Chu-Lee-Zhu \cite{CLZ25} establish the structure theorem for compact K\"ahler manifolds with semi-positive mixed curvature $\mC_{a,b}\geq0$ by generalizing Cheeger-Gromoll's splitting theorem \cite{CG71}. Note that Theoorem \ref{main-thm-partially} provide a complement for their result in the sense of rational connectedness. See a detailed discussion in Section \ref{section-application}.
    \end{remark}
    \noindent The basic strategy of Theorem~\ref{main-thm-partially} builds upon the framework established in \cite{CH19,matsumura2022} for studying structure theorems of projective manifolds with nef anti-canonical bundles or semi-positive holomorphic sectional curvature via H\"oring's foliation theory for rationally connected leaves (Lemma~\ref{lem-morphism}) with the classical Ehresmann's theory (Lemma~\ref{lem-integrability}). The key innovation lies in establishing a Bochner-type integral inequality (Proposition \ref{prop-integral-inequality-1}) for the Hermitian ratio associated with $f^*K_Y$, which requires new considerations beyond our previous work for semi-positive holomorphic sectional curvature \cite{ZZ25}.
    
    We conclude the introduction by outlining the core method of this work. The proofs of both main results—Theorem \ref{main-thm-BC-p-1} and Theorem \ref{main-thm-partially}—rest on a detailed analysis of the Hermitian ratio $\psi$ attached to a pseudo-effective subsheaf $\mathcal{F}\subset\Omega_X^p$, which is developed in Sections \ref{Hermitian ratio} and \ref{section-bochner-integral}. Central to our approach are Proposition \ref{prop-Bochner} and Proposition \ref{prop-integral-inequality-1}. Certain parts of this framework, notably Proposition \ref{prop-Bochner}(1) and Lemma \ref{lem-mixedcurvature-geometriccondition}, are established specifically for the choice $\mathcal{F}:=f^*K_Y$. Their proofs rely on a careful construction of currents associated to a plurisubharmonic function $\varphi$, as presented in Section \ref{currents-definition}, with particular attention paid to the object $e^{\varphi}\im\partial\varphi\wedge\bar\partial\varphi$. A key step is Lemma \ref{lem-current-vanishing}, which reproduces the results in \cite{Demailly02} and \cite{Wu21}.

    \vspace{0.2cm}
    
    {\em This paper is organized as follows.} In section \ref{Preliminaries}, we clarify some basic statements. In section \ref{section-BC}, we study the relationship between BC-$p$ positivity and rational connectedness. Section \ref{section-structure} is devoted to establish the structure theorems for immediate semi-positive curvature conditions.

    \subsection*{Acknowledgements} 
    The authors would like to thank Lei Ni for his interest and useful comments on this work, Kai Tang for pointing out a misunderstanding in the first version of Lemma \ref{lem-k-k+1}, and the referee for identifying several gaps in the original proofs of Proposition \ref{prop-Bochner} and Proposition \ref{prop-integral-inequality-1}. The first author also wishes to thank Shin-ichi Matsumura for his suggestions on related topics. The research was supported by the National Key R and D Program of China 2020YFA0713100. The  authors are partially supported by NSF in China No.12141104, 12371062 and 12431004.
	
    \section{Basic materials}\label{Preliminaries}
    
    In this section, we recall some basic definitions and facts of currents associated to a plurisubharmonic function, MRC fibrations and positivity of vector bundles.
    
    \subsection*{Notations and Conventions} For a Hermitian metric $g$ on a complex manifold $X$, we denote by $\omega_g$ its associated K\"ahler form; the subscript $g$ will often be omitted when no confusion can arise; $W^{k,p}(X), W^{k,p}_\loc(X)$ refer to the Sobolev spaces with respect to the volume measure induced by $\omega^n$, we emphasize this point because the other positive radon measures shall appear when proceeding positive currents. 
    
    To distinguish the pairing $\langle T,u \rangle$ of between a current $T$ and a test form $u$, we denote by $\langle\cdot,\cdot\rangle_h$ be the pairing induced by a Hermitian metric $h$ of a vector bundle $E$ and $\{\cdot,\cdot\}_h$ be the pairing of $E$-valued forms induced by $\langle\cdot,\cdot\rangle_h$ together with wedge product of forms (see e.g. \cite[Pages 262]{demaillycomplex}).
    
    The product of forms with bounded Borel measurable coefficients and currents with complex Radon measurable coefficients can be well-defined. For a $(r,s)$-form $\eta$ with bounded Borel measurable coefficients and a current $T$ of bidegree $(p,q)$, $T\wedge \eta$ refers to the current of bidegree $(p+r,q+s)$ defined by
    \[
    \langle T\wedge\eta, u \rangle := \langle T, \eta\wedge u \rangle, \
    \forall u \in C^\infty_0\bigl( X, \Lambda^{n-(p+r),n-(q+s)} T_X^* \bigr),
    \]
    which is well-defined because each coefficient of $\eta\wedge u$ are locally integrable with respect to each coefficient of $T$. In particular, $T\wedge\eta$ is well-defined when $T$ is a positive current. 

    \subsection{Currents associated to a plurisubharmonic function}\label{currents-definition}    
    For basic properties of currents and plurisubharmonic functions, we refer to \cite{Demailly02,GZ17}. Let $U$ be an open subset of $\C^n$ and  $PSH(U)$ denote the set of plurisubharmonic functions in $U$. It's well-known that
    \begin{lemma}[cf. {\cite[Proposition 2.20]{GZ17}}]
       If $\varphi \in PSH(U)$ and $\varphi \not\equiv -\infty$, then $\varphi \in L^1_\mathrm{loc}(U)$ and $\im\partial\bar\partial\varphi$ is a closed positive current.	
    \end{lemma}
    
    \noindent Since $e^{\varphi}$ is locally bounded and Borel measurable, the product $e^{\varphi}\im\partial\bar\partial\varphi$ is consequently a well-defined positive current. Here, we focus on studying the approximation properties of currents associated to $\varphi$ and giving a rigorous definition to the quantity $e^{\varphi}\im\partial\varphi\wedge\bar\partial\varphi$, which plays a crucial role in our arguments and is not defined priori. We begin with the following integrability properties.
    \begin{lemma}[cf. {\cite[Theorem 4.18]{GZ17}}]\label{lem-convergent-1}
    	$ PSH(U)\cap L_\loc^1(U)\subset W^{1,p}_\loc(U)$ for all $1\leq p<2$. If $\{\varphi_j\}$ is a sequence of functions in $PSH(U)$ such that $\varphi_j\rightarrow\varphi$ in $L_\loc^1(U)$ for some $\varphi\in PSH(U)$, then
    	\begin{itemize}    		
    		\item[(1)] $\{\varphi_j\}$ are locally uniformly bounded from above;
    		\item[(2)] $\varphi_j\rightarrow \varphi$ in $W_\loc^{1,p}(U)$.
    	\end{itemize}
    \end{lemma}
    \noindent This result is optimal, since $\log|z_1|$ is an unbounded plurisubharmonic function whose gradient does not belongs to $L_\loc^2(U)$. Nevertheless, the gradient of a plurisubharmonic function is in fact locally $L^2$-integrable with respect to the weight $e^{\varphi}$. Let us recall the standard approximations by smooth plurisubharmonic functions before the proof. 
    
    \vspace{0.1cm}
    
    Let $\{\rho_\epsilon\}_{\epsilon>0}$ be a family of smoothing kernels, i.e., $\rho_\epsilon$ is a smooth function compactly supported in $B(0,\epsilon)$ with $\int_{\C^n}\rho_\epsilon=1$ and $\rho_\epsilon$ approximates the Dirac measure at the origin. Set $U_\epsilon:=\{x\in U:d(x,\p U)>\epsilon\}$, we have
    \begin{lemma}[cf. {\cite[Proposition 1.42]{GZ17}}]\label{lem-smooth-approximation}
    	Let $\varphi\in PSH(U)\cap L^1_\loc(U)$, then $\varphi_\epsilon:=\varphi*\rho_\epsilon\in PSH(U_\epsilon)\cap C^\infty(U_\epsilon)$ decreases to $\varphi$ as $\epsilon$ decreases to $0^+$ and $\varphi_\epsilon\rightarrow \varphi$ in $L^1(K)$ for any relatively compact subset $K$.
    \end{lemma}
    
   The approximation properties listed below suffice for our subsequent computations.
    
    \begin{proposition}\label{prop-convergent}
    	Let $\varphi\in PSH(U)\cap L_\loc^1(U)$. The following statements hold for any relatively compact open subset $K\subset\subset U$ and $c>0$.
    	\begin{itemize}
    		\item[(1)] $e^{c\varphi_\epsilon}\rightarrow e^{c\varphi}$ in $W^{1,1}(K)$; in particular, $e^{c\varphi_\epsilon}\p\varphi_\epsilon\rightarrow e^{c\varphi}\p\varphi$ and $\p e^{\varphi}=e^{\varphi} \p\varphi$ in $L^1(K)$.
    		\item[(2)] $e^{c\varphi_\epsilon}\im(\pp e^{\varphi_\epsilon})\rightharpoonup e^{c\varphi}\im\pp(e^\varphi)$ in the weak topology of currents on $K$.
    		\item[(3)] $e^{c\varphi}\in W^{1,2}_\loc(U)$; in particular, $e^{c\varphi}\im\p\varphi\wedge\bp\varphi$ is a well-defined positive current via integration. Moreover, $e^{c\varphi_\epsilon}\p\varphi_\epsilon\wedge\bp\varphi_\epsilon$ converges and
    		$$e^{c\varphi}\im\p\varphi\wedge\bp\varphi\leq \lim\limits_{\epsilon\rightarrow 0} e^{c\varphi_\epsilon}\im\p\varphi_\epsilon\wedge\bp\varphi_\epsilon$$
    		in the sense of currents on $K$.
    		\item[(4)] 
    		$\im\pp e^{\varphi}-e^\varphi\im\p\varphi\wedge\bp\varphi$
    		is a positive current.
    	\end{itemize}
    \end{proposition}
    
   For all sufficiently small $\epsilon > 0$, the inclusion $K \subset \Omega_\epsilon$ holds; consequently, the limit operations discussed above are valid. Moreover, by rescaling, it is enough to establish the above statements for any fixed constant $c$. The constant is retained in the statement to facilitate direct reference and application.
    
    \begin{proof}
       Since $\lim\limits_{\epsilon\to 0} e^{\varphi_\epsilon}=e^{\varphi}$ and the family ${e^{\varphi_\epsilon}}$ is uniformly bounded on $K$, Lebesgue's dominated convergence theorem gives convergence $e^{\varphi_\epsilon} \to e^{\varphi}$ in $L^1(K)$. By Lemma \ref{lem-convergent-1}, it follows that $e^{\varphi_\epsilon}\partial\varphi_\epsilon = \partial e^{\varphi_\epsilon} \to \partial e^{\varphi}$ in $L^p(K)$ for every $1\leq p<2$. Meanwhile, we have $\partial\varphi_\epsilon \to \partial\varphi$ in $L^1(K)$; passing to a subsequence ensures almost‑everywhere convergence, and consequently $e^{\varphi_\epsilon}\partial\varphi_\epsilon \to e^{\varphi}\partial\varphi$ almost everywhere. The Vitali convergence theorem then gives $e^{\varphi_\epsilon}\partial\varphi_\epsilon \to e^{\varphi}\partial\varphi$ in $L^1(K)$ (after passing to a further subsequence). We therefore conclude that $e^{\varphi}\partial\varphi = \partial e^{\varphi}$ in $L^1(K)$, which establishes statement (1). Statement (2) follows directly from the result of Bedford–Taylor \cite[Theorem 2.1]{BT82}.

       Now let us show (3). Applying (2) gives that
       $$e^{\varphi_\epsilon}\im\p\varphi_\epsilon\wedge\bp\varphi_\epsilon=2\im\pp e^{\varphi_\epsilon}-4e^{\frac{\varphi_\epsilon}{2}}\im\pp e^{\frac{\varphi_\epsilon}{2}}$$
       converges in the sense of currents. Note that
       \begin{equation}\label{equa-smooth-sum}
       	  \im\pp e^{\varphi_\epsilon}=e^{\varphi_\epsilon}\im\p\varphi_\epsilon\wedge\bp\varphi_\epsilon+e^{\varphi_\epsilon}\im\pp\varphi_\epsilon\geq e^{\varphi_\epsilon}\im\p\varphi_\epsilon\wedge\bp\varphi_\epsilon.
       \end{equation}
       It follows that $e^{\varphi_\epsilon}\im\partial\varphi_\epsilon\wedge\bar\partial\varphi_\epsilon$ is uniformly bounded in $L^1(K)$. This can be verified by wedging with $\chi(\im\partial\bar\partial|z|^2)^{n-1}$, since the coefficients are then dominated by the trace measure; here $\chi$ is a smooth function that equals $1$ on a neighborhood of $K$ and has compact support in $U$. Moreover, the convergence $\partial e^{\frac{1}{2}\varphi_\epsilon} \to \partial e^{\frac{1}{2}\varphi}$ holds in $L^{1}(K)$ and therefore in Lebesgue measure; consequently, by passing to a suitable subsequence, $\{e^{\varphi_\epsilon}\im\p\varphi_\epsilon\wedge\bp\varphi_\epsilon\}$ converges almost everywhere. Applying Fatou's lemma, we obtain for every positive smooth $(n-1,n-1)$ forms $u$ with compact support in an open neighborhood of $K$, it holds that
       \begin{equation}\label{equa-leq-fatou}
       	0\leq \int_K e^{\varphi}\im\p\varphi\wedge\bp\varphi\wedge u\leq \lim\limits_{\epsilon\rightarrow 0} \int_X e^{\varphi_\epsilon}\im\p\varphi_\epsilon\wedge\im\bp\varphi_\epsilon\wedge u.
       \end{equation}
       Since the right-hand side is finite, we conclude that $\p e^{\frac{1}{2}\varphi}\in L^2(K)$ by taking $u=\chi(\im\pp|z|^2)^{n-1}$ where $\chi$ is chosen as before, which establishs (3). Finally, using the statement (2) together with \eqref{equa-smooth-sum} and \eqref{equa-leq-fatou}, we have
       \begin{equation}
       	   \big\langle \im\pp e^{\varphi}-e^\varphi\im\p\varphi\wedge\bp\varphi,u \big\rangle \geq \lim\limits_{\epsilon\rightarrow 0}\int_X e^{\varphi_\epsilon}\im\pp\varphi_\epsilon\wedge u\geq0,
       \end{equation}
       which verifies (4) and thus completes the proof.
    \end{proof}
     
    \subsection{MRC fibrations}
    First, let us recall the definitions of uniruledness and rational connectedness, both of which play a crucial role in birational geometry (see, e.g., \cite{De2001}).
    \begin{definition}
    	Let $X$ be a compact K\"ahler manifold. We say that
    	\begin{itemize}
    		\item $X$ is uniruled if for every general point $x\in X$, there exists a rational curve containing $x$;
    		\item $X$ is rationally connected if any two points can be connected by a rational curve.
    	\end{itemize}
    \end{definition}
    
    \begin{definition}
    	For a line bundle $L$ over a compact complex manifold $X$, we say that $L$ is pseudo-effective if it admits a singular metric $h$ with semi-positive curvature, i.e., the weight
    	of $h$ with respect to any trivialization coincides with some plurisubharmonic function almost everywhere.
    \end{definition}
    
    The result of \cite{BDPP} was extended to compact K\"ahler manifolds in the recent preprint of Ou \cite{Ou25}, which states that a compact K\"ahler manifold $X$ is uniruled if and only if $K_X$ is not pseudo-effective.  The existence of the MRC ﬁbration of a compact K\"ahler manifold follows, as in the projective case (\cite{Cam92,KMM}), from the existence of a quotient map for covering families (\cite[Theorem 2.6]{Cam04}). Combining these results, we obtain
      \begin{lemma}\label{lem-MRCfibration}
		Let $X$ be a compact K\"ahler manifold. There exists a MRC fibration $f:X\rightarrow Y$ of $X$. In particular, it is an almost holomorphic map $f:X\dashrightarrow Y$ (that is meromorphic map whose general fibres are compact) to a compact K\"ahler manifold $Y$ with $K_Y$ pseudo-effective such that $\dim X-\dim Y\geq0$ and the general fibres are rationally connected.
	\end{lemma}
    For reader's convenience, we provide a proof of Lemma \ref{lem-MRCfibration}.
    \begin{proof}
      Consider a MRC fibration $f:X\dashrightarrow Y$ of $X$ to a compact K\"ahler manifold $Y$
    	\begin{itemize}
    		\item [(1)] $f$ is an almost holomorphic map;
    		\item [(2)] General fibres are rationally connected;
    		\item [(3)] There is no horizontal rational curve passing through a general point in $X$.
    	\end{itemize}
    	In particular, $\dim X-\dim Y>0$ as $X$ is uniruled. By \cite{GHS}, any rational curves in $Y$ can be lifted into $X$ and thus $Y$ is not uniruled by the statement (3).  Then it follows from Ou's recent result \cite{Ou25} that $K_Y$ is pseudo-effective.
    \end{proof}
    
    \begin{lemma}\label{lem-pseudo-effective}
    	Let $f:X\dashrightarrow Y$ be a dominant meromorphic map between compact complex manifolds $X,Y$. For simplicity, we will always denote $(f|_{X\setminus Z})^*K_Y$ by $f^*K_Y$, where $Z$ is the indeterminacy locus of $f$.  If $K_Y$ is pseudo-effective, then the unique reflexive extension $\mathcal{F}$ of $f^*K_Y$ is pseudo-effective.
    \end{lemma}
    
    \begin{proof}
    	Since $K_Y$ is pseudo-effective, it supports a singular Hermitian metric $h$, then $\big(f^*K_Y,f^*h)$ is a pseudo-effective line bundle on $X\setminus Z$. Let $\mathcal{F}$ be the trivial reflexive extension of $f^*K_Y$, then $\mathcal{F}$ is a line bundle and is pseudo-effective since each plurisubharmonic function $\varphi_\alpha$ on $X\setminus Z$ can be uniquely extended to $X$ by the extension theorem of Grauert and Remmert (\cite[page 181]{GR56}) since $\mathrm{codim}_XZ\geq2$.
    \end{proof}
    
    We will use the following lemma to construct a MRC fibration that is holomorphic.
    \begin{lemma}(\cite[Corollary 2.11]{horing07})\label{lem-morphism}
        Let $X$ be a compact K\"ahler manifold and $V\subset T_X$ be an integrable subbundle such that one leaf is compact and rationally connnected. Then there exists a smooth submersion $X\rightarrow Y$ onto a compact K\"ahler manifold $Y$ such that $T_{X/Y}=V$.
    \end{lemma}
    
    \begin{definition}[Locally constant fibration, {c.f. \cite[Definition 2.6]{matsumurasurvey}}]\label{defn-locallyconstant}
    	Let $f:X\rightarrow Y$ be a fibration (i.e., a surjective holomorphic map with connected fibres) between compact K\"ahler manifolds.
    	\begin{itemize}
    		\item[(a)] $f$ is said to be a {\em locally trivial fibration} if for each point of $Y$ there is an open neighborhood $U\subset Y$ and a holomorphic map $f^{-1}(U)\cong U\times F$.
    		\item[(b)] Suppose that $f$ is a locally trivial fibration with fibre $F$. The fibration $f$ is said to be a {\em locally constant fibration} if there exists a representation $\rho:\pi_1(Y)\rightarrow \Aut(F)$
    		from the fundamental group $\pi_1(Y)$ of $Y$ into the group of automorphisms of $F$, such that
    		$$X\cong (Y_{\mathrm{univ}}\times F)/\pi_1(Y),$$
    		where $Y_{\mathrm{univ}}$ is the universal cover of $Y$ and $\gamma\in \pi_1(Y)$ acts on $Y_{\mathrm{univ}}\times F$ by
    		$$\gamma\cdot(y,p)=(\gamma\cdot y,\rho(\gamma)\cdot p).$$
    		In particular, we have the following commutative diagram:
    		\[
    		\begin{tikzcd}
    			F_{\mathrm{univ}} & X_{\mathrm{univ}}=Y_{\mathrm{univ}}\times F_{\mathrm{univ}} \arrow[l,"\mathrm{pr}_2"'] \arrow[r] \arrow[rd,"\mathrm{pr}_1"'] \arrow[rr,bend left=10,"\mu"] & Y_{\mathrm{univ}}\times F \arrow[r] \arrow[d] & (Y_{\mathrm{univ}}\times F)/\pi_1(Y)\cong X \arrow[d,"f"] \\
    			& & Y_{\mathrm{univ}} \arrow[r,"\pi"] & Y.
    		\end{tikzcd}
    		\]
    	\end{itemize}
    \end{definition}   
    The classical Ehresmann's theorem can be applied to to verify a holomorphic map is a locally constant fibration.
    \begin{lemma}(c.f. \cite[Theorem 3.17]{horing07})\label{lem-integrability}
        Let $f:X\rightarrow Y$ be a submersion of manifolds with an integrable connection, i.e, an integrable subbundle $W$ of $T_X$ such that $T_X=W\oplus T_{X/Y}$. Then $f:X\rightarrow Y$ is a locally constant fibration such that $\mu^*W\simeq \mathrm{pr}_1^*(T_{Y_{\mathrm{univ}}})$ and $\mu^* T_{X/Y}\cong \mathrm{pr}_2^* T_{F_{\mathrm{univ}}}$.
    \end{lemma}
 
    \vspace{0.1cm}
    
    \subsection{Positivity of vector bundles}\label{section-positivity}
    In this subsection, we compare BC-$p$ positivity with other curvature positivity conditions. We will always assume that $(E,h)$ is a Hermitian holomorphic vector bundle on a complex manifold $X$. Let $\{\frac{\p}{\p z^i}\}$ be the local holomorphic coordinates on $X$ and $\{e_\alpha\}$ be an unitary frame of $E$. Then the Chern curvature tensor $R^{(E,h)}$ has components
    $$R_{i\bar{j}k\bar{l}}=-\frac{\p^2 h_{k\bar{l}}}{\p z^i \p \bar{z}^j}+h^{p\bar{q}}\frac{\p h_{k\bar{q}}}{\p z^i}\frac{\p h_{p\bar{l}}}{\p \bar{z}^j}.$$
    \noindent Fix another Hermitian metric $g$ on $X$, the {\em mean curvature} of $(E,h)$ with respect to $g$ is defined by
    $$\tr_g R^{(E,h)}=g^{i\bar{j}}R_{i\bar{j}k\bar{l}}.$$ 
    \begin{definition}\label{defn-vectorbundle-Hermitianmetric}
    	BC-$p$-positivity has been introduced in Definition \ref{defn-BC-positivity}. We say that 
    	\begin{itemize}
    		\item $(E,h)$ is RC-positive if it is BC-$1$ positive.
    		\item $(E,h)$ is {\em mean curvature positive} if there exists some Hermitian metric $g$ on $X$ such that $\tr_g R^{(E,h)})$ is positive definite on $E$.
    		\item $(E,h)$ is uniformly RC-positive if for any $x\in X$, there exists some $v\in T_{X,x}$ such that
    		$R^{(E,h)}(v,\bar{v},\cdot,\cdot)$ is positive definite. This is introduced in \cite{yang2019}
    	\end{itemize}
    \end{definition}
    Clearly, mean curvature positivity of $(E,h)$ implies the mean curvature positivity of $(\Lambda^pE,\Lambda^p h)$, and thus RC-positivity of $(\Lambda^pE,\Lambda^ph)$ for every $p\geq1$. Recall that uniformly RC-positivity implies the mean curvature positivity (see e.g. \cite[Proposition 3.6]{LZZ21}).
    \begin{proposition}\label{prop-relationship} We have the following relationship for $(E,h)$ for every $p\geq1$:
    	\[
    	\begin{tikzcd}
    		\text{uniformly RC-positive} \arrow[r,Rightarrow] & \text{mean curvature positive} \arrow[ld,Rightarrow] &  \\
    		(\Lambda^pE,\Lambda^ph)\text{ is RC-positive} \arrow[r,Rightarrow] & (E,h)\text{ is BC-$p$-positive} 
    	\end{tikzcd}
    	\]
    \end{proposition}
    The last implication follows from the following basic statement:
    \begin{lemma}\label{lem-BC-RC}
    	$(E,h)$ is BC-$p$ positive for some $1\leq p\leq n$ if and only if for any $x\in X$ and any nonzero decomposable section $u=e_1\wedge\cdots\wedge e_p\in \Lambda^p E_x$, there exists some $v\in T_{X,x}$ such that
    	$$R^{(\Lambda^hE,\Lambda^ph)}(v,\bar{v},u,\overline{u})>0.$$
    \end{lemma}
    
    \begin{proof}
    	Suppose that $u=e_1\wedge\cdots\wedge e_p\in \Lambda^p E_x$ is a nonzero decomposable section at one point $x$. Set $\Sigma=\spaned\{e_1,\cdots,e_p\}$. Since $u$ is nonzero, $\dim \Sigma=p$. Choose an unitary frame $\{\widehat{e_1},\cdots,\widehat{e}_p\}$. Then $u=t \widehat{e_1}\wedge \cdots\wedge\widehat{e_p}$ for some $t\in\C^*$ and for any $v\in T_{X,x}$, it holds that
    	\begin{align*}
    		R^{(\Lambda^hE,\Lambda^ph)}(v,\bar{v},u,\overline{u})=&|t|^2R^{(\Lambda^hE,\Lambda^ph)}(v,\bar{v},\widehat{e_1}\wedge\cdots\wedge \widehat{e_p},\overline{\widehat{e_1}\wedge\cdots \widehat{e_p}})\\
    		=&|t|^2\sum\limits_{i=1}^p R^{(E,h)}(v,\bar{v},\widehat{e_i},\overline{\widehat{e_i}}),
    	\end{align*}
        which completes the proof building on the definition of RC-positivity and BC-$p$ positivity,
    \end{proof}
    
    At one point, from a linear algebra perspective, the reverse directions of the relationships above are all incorrect. However, for tangent bundles, their entire positivity will imply the rational connectedness when $X$ is K\"ahler. Actually, in \cite{yang2018,yang2019}, based on Campana-Demailly-Peternell's criterion for rational connectedness, X. Yang proved that:
    \begin{theorem}[X. Yang]
    	Let $X$ be a compact K\"ahler manifold of complex dimension $n$, if $\Lambda^pT_X$ is RC-positive for every $1\leq p\leq n$, then $X$ is rationally connected.
    \end{theorem}
    
    In this paper, we shall prove the rational connectedness under the assumption of BC-$p$ positivity for all $1\leq p\leq n$ in Section \ref{section-BC}. Conversely, a recent work of Li-Zhang-Zhang established that the rational connectedness implies the mean curvature positivity of the tangent bundle based on constructing $L^p$-approximate critical Hermitian structures in \cite{LZZ21}, thus give a differential geometric criterion for rational connectedness.
    
    \begin{theorem}[Li-Zhang-Zhang]\label{thm-RC-meancurvature}
    	A compact K\"ahler manifold is rationally connected if and only if the tangent bundle is mean curvature positive.
    \end{theorem}

	\section{BC-positivity and Rational connectedness}\label{section-BC}
	In this section, we shall establish Theorem \ref{main-thm-BC-p-1} and provide several applications. The key ingredient is deriving the strict subharmonicity of the hermitian ratio along some direction.
	
	\subsection{Bochner-type formula}\label{Hermitian ratio}
	  We first introduce the Hermitian ratio of a pseudo-effective subsheaf (see \cite{campana2015} for a similar consideration) and then establish its associated Bochner-type formula. 
	  
	  Suppose that $(X,g)$ is a compact Hermitian manifold of dimension $n$ and $\mF$ is a pseudo-effective invertible subsheaf of $\Omega^p_X:=\mO(\Lambda^{p}T_X^*)$. Then there exist a finite covering $\{U_\alpha\}$ consists of coordinate neighborhoods and local holomorphic $(p,0)$-forms $\eta_\alpha$ generating $\mathcal{F}|_{U_\alpha}$, $\mathcal{F}$ is associated with the \u{C}ech cocycle $b_{\alpha\beta}$ in $\mathcal{O}_X^*(U_\alpha)$ such that $\eta_\beta = b_{\alpha\beta} \eta_\alpha$ and there exists a singular Hermitian metric $h=e^{-\varphi}$ of $\mathcal{F}$ defined by a collection of plurisubharmonic functions $\varphi_\alpha \in PSH(U_\alpha)\cap L_{\loc}^1(U_\alpha)$ such that $e^{-\varphi_\alpha}=|\eta_\alpha|^2_h$.
	 
	 \begin{definition}\label{defn-psi} For any pseudo-effective subsheaf $\mF\subset \Omega_X^p$, its Hermitian ratio is locally defined by
	 	$$\psi:=e^{\varphi_\alpha}|\eta_\alpha|_{\Lambda^pg^*}^2.$$ 
	 One readily check that $\psi$ is indeed global.
	 \end{definition} 
	 
	 \begin{remark}\label{remark-viewpoint}
	 	There is a reversible identification of $\mathcal{F}$ with a global holomorphic section $\sigma = \eta_\alpha \otimes e_\alpha$ of $H^0(U, \Omega_X^p \otimes \mathcal{F}^{-1})$, where $e_\alpha$ is the dual basis of $\mathcal{F}^{-1}|_{U\alpha}$ satisfying $e_\alpha(\eta_\alpha)=1$; in particular $|e|_{h^{-1}}^2=e^{\varphi_\alpha}$. Under this identification, $\psi$ equals the squared norm $|\sigma|_H^2$ with respect to the singular metric $H=\Lambda^pg^*\otimes h^{-1}$.
	 \end{remark}
	 
	 The key ingredient of proving Theorem \ref{main-thm-BC-p-1}  is the following technical result.
	 
	 \begin{proposition}\label{prop-subharmonic}
	 	Let $f:X\dashrightarrow Y$ be a dominant meromorphic map between compact complex manifolds with $p:=\dim Y\leq \dim X$, and $\mF$ be the reflexive extension of $f^*K_Y$. If $T_X$ is BC-$p$-positive and $K_Y$ is pseudo-effective, then for any $x\in X$ such that $\psi(x)\neq0$, there exists $0\neq v\in T_{X,x}$ such that $\log\psi|_{L_v}$ is strictly subharmonic on a neighborhood of $x$, where $L_v:=v\cdot \C$ is a line defined by the local holomorphic coordinates centered at $x$.
	 \end{proposition}
	
	 The proof proceeds by computing $\im\partial\bar\partial\psi$ in the sense of currents. A key observation is that when $\mathcal{F}$ is the reflexive extension of $f^*K_Y$, the form $\eta_\alpha$ is decomposable at every point of $U_\alpha$; in this situation, BC positivity already suffices to yield the positivity of $\im\partial\bar\partial\log\psi$ in certain directions. This property may fail for a general pseudo-effective subsheaf $\mathcal{F}$. We now state the precise statement, which is also essential to the proof of Proposition \ref{prop-integral-inequality-1}.
	 
	 \begin{proposition}\label{prop-Bochner}
	 	Let $(X,g)$ be a compact Hermitian manifold and $\mF$ be a pseudo-effective invertible subsheaf of $\Omega_X^p$. The following hold.
	 	\begin{itemize}
	 		\item[(1)] When $\mF$ coincides with the reflexive extension of $f^*K_Y$ for a dominant meromorphic map $f:X\dashrightarrow Y$, then
	 		$\eta_\alpha$ is decomposable at every point of $U_\alpha$, i.e., for any $x\in U_\alpha$, there exists some unitary coframe $\{e_i\}_{i=1}^p$ of $T_{X,x}^*$ such that $\eta_\alpha(x)=t_{\alpha,x} e_1\wedge \cdots \wedge e_p$ for some $t_{\alpha,x}\in\C$. In particular, $|t_{\alpha,x_0}|^2=|\eta_\alpha(x_0)|_{\Lambda^pg^*}^2$ and
	 		\begin{equation}\label{equa-Curvature}
	 			-\langle R^{\Lambda^pg^*}(u,\bar{u})\eta_\alpha,\overline{\eta_\alpha}\rangle_{\Lambda^pg^*}=|\eta_\alpha|^2\sum\limits_{i=1}^p R_{u\bar{u}i\bar{i}}
	 		\end{equation}
	 		for all $u\in T_{X,x}$, where $R^{\Lambda^pg^*}$ (resp. $R$) denotes the Chern curvature tensor of $\Lambda^pg^*$ (resp. $g$) and  $R_{u\bar{u}i\bar{i}}=R(u,\bar{u},e_i^*,\overline{e_i^*})$.
	 		\item[(2)] For any $x\in U_\alpha$ and all $u\in T_{X,x}$, we have the following Bochner-type formulas
	 		\begin{equation}\label{equa-Bochner-log}
	 			\im\pp\log|\eta_\alpha|_{\Lambda^pg^*}\geq-\frac{\{\im R^{\Lambda^pg^*}\eta_\alpha,\overline{\eta_\alpha}\}_{\Lambda^pg^*}}{|\eta_\alpha|^2_{\Lambda^p g^*}}
	 		\end{equation}
	 		and
	 		\begin{equation}\label{equa-Bochner}
	 			\begin{split}
	 				\im\p\bp\psi=&e^{\varphi_\alpha}\im\{D'_H\eta_\alpha,\overline{D'_H\eta_\alpha}\}_{\Lambda^pg^*}
	 				-e^{\varphi_\alpha}\{\im R^{\Lambda^pg^*}\eta_\alpha,\overline{\eta_\alpha}\}_{\Lambda^pg^*}\\
	 				&+\left(\im\pp e^{\varphi_\alpha}-e^{\varphi_\alpha}\im\p\varphi_\alpha\wedge\bp\varphi_\alpha\right)|\eta_\alpha|_{\Lambda^pg^*}^2
	 			\end{split}
	 		\end{equation}
	 	    in the sense of currents, where $D_H'\eta_\alpha=D'\eta_\alpha+\p\varphi_\alpha\otimes\eta_\alpha$ and $D'$ denotes the $(1,0)$-part of the Chern connection of $g$.
	 	\end{itemize}
	 \end{proposition}
	 
	 From the viewpoint explained in Remark \ref{remark-viewpoint}, $D'_H$ can be viewed as the Chern connection of the singular Hermitian metric $H=\Lambda^pg^*\otimes h^{-1}$ and so \eqref{equa-Bochner} is just the standard Bochner-type formula of $\im\pp|\sigma|^2_H$ for $\sigma=\eta_\alpha\otimes e_\alpha\in H^0(X,\Omega_X\otimes\mF^{-1})$ when $h$ is smooth.
	 
	 \begin{proof}[\bf Proof of Proposition \ref{prop-Bochner} (1)]

	 	To avoid any misunderstandings, we will emphasize when the local coordinates have been re-selected, when necessary. Denote the indeterminacy locus of $f$ by $Z$. First, we consider the case $x_0\in U_\alpha\setminus Z$. Let $\{V_\beta\}$ be a finite cover of $Y$ consists of coordinates neighborhood $(\w_1\cdots,\w_p)$. Fix $x_0\in U_\alpha\setminus Z$, there exists some $\beta$ such that $x_0\in f^{-1}(V_\beta)$. Since $\mF$ coincides with $f^*K_Y$ on $X\setminus Z$, $f^*(d\w_1\cdots d\w_p)=f^*d\w_1\wedge\cdots\wedge f^*d\w_p$ generates $\mF|_{(U_\alpha\setminus Z)\cap f^{-1}(V_\beta)}$ and thus there exists some $t_{\alpha,\beta}\in\mO^*\left((U_\alpha\setminus Z)\cap f^{-1}(V_\beta)\right)$ such that $\eta_\alpha=t_{\alpha,\beta}f^*d\w_1\wedge\cdots\wedge f^*d\w_p.$ Re-choose local holomorphic coordinates $(z_1\cdots,z_n)$ centered at $x_0$ and $(\w_1,\cdots,\w_p)$ centered at $f(x_0)$ such that, at point $x_0$, we have
	 	$\w=\im\sum\limits_{i=1}^pdz^i\wedge d\bar{z}^i$ and $\frac{\p f^j}{\p z^i}=r_i\delta_i^j$
	 	for some real numbers $r_1,\cdots,r_p$, where $f^j=\w^j\circ f$. Then
	 	\begin{equation}\label{equa-form-localexpression}
	 		\eta_\alpha(x_0)=t_{\alpha,\beta}(x_0)r_1\cdots r_p dz^1\wedge\cdots \wedge dz^p
	 	\end{equation}
	 	over $U_\alpha\setminus Z$ and we conclude that $\eta_\alpha$ is decomposable at a generic point of $U_\alpha$. It remains to consider the case $x_0\in U_\alpha\cap Z$. Consider the map $\rho:\Omega_X^1\rightarrow \Omega_X^{p+1}$ on $U_\alpha$ defined by $v\mapsto v\wedge\eta_\alpha$. It can be easily seen that $\ker\rho$ has rank at most $p$ in $U_\alpha$; and has rank $p$ at one point $x$ if and only if $\eta_\alpha(x)$ is non-zero and is decomposable (see e.g. \cite[Lemma 3.32]{CHLMSSX25}). Suppose that $\eta_\alpha(x_0)$ is non-zero and nondecomposable. Then $\rank(\ker\rho(x_0))<p$ and so $\rank(\ker\rho(x))<p$, which contradicts to that we have shown. 
	 	
	 	Consequently, $\rho_x$ is decomposable at every point and we can write $\eta_\alpha(x)=t_{\alpha,x}e_1\wedge\cdots\wedge e_p$ for some unitary coframe $\{e_i\}$ of $T_{X,x}$ and $t_{\alpha,x}\in\R$. Then, at $x$, for any $u\in T_{X,x}$, it holds that
	 	\begin{align*}
	 		&-\langle R^{(\Lambda^pT_X^*,\Lambda^pg^*)}(u,\bar{u})\eta_\alpha, \overline{\eta_\alpha}\rangle_{\Lambda^pg^*}\\
	 		&=-\langle R^{(\Lambda^pT_X^*,\Lambda^pg^*)}(u,\bar{u})\left(t_{\alpha,x_0}e_1\wedge\cdots\wedge\cdots e_p\right),\overline{t_{\alpha,x_0}e_1\wedge\cdots\wedge\cdots e_p}\rangle_{\Lambda^pg^*}\\
	 		&=|\eta_\alpha|^2\sum\limits_{i=1}^p R(u,\bar{u},e_i^*,\overline{e_i^*}),
	 	\end{align*}
	 	where the last equality follows from $|\eta_\alpha(x_0)|^2=|t_{\alpha,x_0}|^2$. The proof is complete.
	 \end{proof}
	 
	 \begin{proof}[\bf Proof of Proposition \ref{prop-Bochner} (2)]
	 	 It suffices to work on $U_\alpha$. Note that $\log|\eta_\alpha|_{\Lambda^pg^*}^2\in L_\loc^1(U_\alpha)$, we have
	 	\begin{align*}
	 		&\im\pp\log|\eta_\alpha|_{\Lambda^pg^*}^2=\lim\limits_{\epsilon\rightarrow 0}\im\pp\log(|\eta_\alpha|_{\Lambda^pg^*}^2+\epsilon)\\
	 		&\geq\lim\limits_{\epsilon\rightarrow 0}\im\bigg(\frac{|\eta_\alpha|_{\Lambda^pg^*}^2\{D'\eta_\alpha,\overline{D'\eta_\alpha}\}_{\Lambda^pg^*}-\{D'\eta_\alpha,\overline{\eta_\alpha}\}_{\Lambda^pg^*}\wedge\overline{\{D'\eta_\alpha,\overline{\eta_\alpha}\}_{\Lambda^pg^*}}}{(|\eta_\alpha|_{\Lambda^pg^*}^2+\epsilon)^2}-\frac{\{ R^{\Lambda^pg^*}\eta_\alpha,\overline{\eta_\alpha}\}_{\Lambda^pg^*}}{|\eta_\alpha|^2_{\Lambda^p g^*}+\epsilon}\bigg)\\
	 		&\geq-\frac{\{\im R^{\Lambda^pg^*}\eta_\alpha,\overline{\eta_\alpha}\}_{\Lambda^pg^*}}{|\eta_\alpha|^2_{\Lambda^p g^*}}
	 	\end{align*}
	 	in the sense of currents, where the last inequality follows from the Cauchy-Schwarz inequality and the Lebesgue's dominated convergence theorem. 
	 	
	 	Now let us verify \eqref{equa-Bochner}. A direct computation on $\psi=e^{\varphi_\alpha}|\eta_\alpha|_{\Lambda^pg^*}^2$ gives that
	 	\begin{align*}
	 		\im\pp\psi
	 		=&\im e^{\varphi_\alpha}\bigg(\p {\varphi_\alpha}\wedge \{\eta_\alpha,\overline{D'\eta_\alpha}\}_{\Lambda^pg^*}+ \{D'\eta_\alpha,\overline{\eta_\alpha}\}_{\Lambda^pg^*}\wedge\bp\varphi_\alpha\bigg)
	 		+e^{\varphi_\alpha}\im\pp|\eta_\alpha|^2_{\Lambda^pg^*}\\
	 		&+|\eta_\alpha|^2_{\Lambda^pg^*}\im\pp e^{\varphi_\alpha}\\
	 		=&\im e^{\varphi_\alpha}\bigg(\p {\varphi_\alpha}\wedge \{\eta_\alpha,\overline{D'\eta_\alpha}\}_{\Lambda^pg^*}+ \{D'\eta_\alpha,\overline{\eta_\alpha}\}_{\Lambda^pg^*}\wedge\bp\varphi_\alpha+|\eta_\alpha|^2_{\Lambda^pg^*}\p\varphi_\alpha\wedge\bp\varphi_\alpha\\
	 		&\ \ \ \ \ \ \ \ \ \ \ \ \ +\{D'\eta_\alpha,\overline{D'\eta_\alpha}\}_{\Lambda^pg^*}\bigg)-e^{\varphi_\alpha}\{\im R^{\Lambda^pg^*}\eta_\alpha,\overline{\eta_\alpha}\}_{\Lambda^pg^*}\\
	 		&\ \ \ \ \ \ \  \ \ \ \ \ \ +|\eta_\alpha|^2_{\Lambda^pg^*}\big(\im\pp e^{\varphi_\alpha}-e^{\varphi_\alpha}\im\p\varphi_\alpha\wedge\bp\varphi_\alpha\big)\\
	 		=&e^{\varphi_\alpha}\im\{D'\eta_\alpha+\p\varphi_\alpha\otimes\eta_\alpha,\overline{D'\eta_\alpha+\p\varphi_\alpha\otimes\eta_\alpha}\}_{\Lambda^pg^*}-e^{\varphi_\alpha}\{\im R^{\Lambda^pg^*}\eta_\alpha,\overline{\eta_\alpha}\}_{\Lambda^pg^*}\\
	 		&+\left(\im\pp e^{\varphi_\alpha}-e^{\varphi_\alpha}\im\p\varphi_\alpha\wedge\bp\varphi_\alpha\right)|\eta_\alpha|_{\Lambda^pg^*}^2,
	 	\end{align*}
	 	where the first equality is valid because $|\eta_\alpha|^2_{\Lambda^pg^*}$ is smooth and $\p e^{\varphi_\alpha}=e^{\varphi_\alpha}\p \varphi_\alpha$ by Proposition \ref{prop-convergent} (1).
	 \end{proof}
	 
	 We conclude this subsection with proving Proposition \ref{prop-subharmonic}.
	 \begin{proof}[\bf Proof of Proposition \ref{prop-subharmonic}]
	 	Since $K_Y$ is pseudo-effective, $\mF$ is pseudo-effective by Lemma \ref{lem-pseudo-effective}.	We may assume that  $(T_X,g)$ if BC-$p$ positive for some Hermitian metric $g$ and $x\in U_\alpha$ for some $\alpha$. Note that $\psi(x)\neq0$ implies that $\eta_\alpha(x)\neq 0$ and $\varphi_\alpha(x)\neq -\infty$. Write $\eta_\alpha=t_{\alpha,x}e_1\wedge\cdots\wedge e_p$ as in the statement (1) of Proposition \ref{prop-Bochner}. BC-$p$-positivity of $T_X$ implies that there exists some $0\neq v\in T_{X,x}$ such that $\sum\limits_{i=1}^p R_{v\bar{v}i\bar{i}}>0$ and thus
	 	$$-\langle R^{(\Lambda^pT_X^*,\Lambda^pg^*)}(v,\bar{v})\eta_\alpha, \overline{\eta_\alpha}\rangle=|\eta_\alpha(x)|^2\sum\limits_{i=1}^p R_{v\bar{v}i\bar{i}}>0.$$
	 	It follows that $(\log|\eta_\alpha|_{\Lambda^pg^*}^2)|_{L_v}$ is strictly subharmonic on a neighborhood of $x$. Together with the decomposition $\log\psi = \varphi_\alpha + \log|\eta_\alpha|_{\Lambda^pg^*}^2$ and the fact that $\varphi_\alpha \in \mathrm{PSH}(U_\alpha)$, this establishes the desired result and finishes the proof.
	 \end{proof}
	
	\subsection{Proof of Theorem \ref{main-thm-BC-p-1} and applications}\label{BC-application}
	Now let us prove Theorem \ref{main-thm-BC-p-1} and provide several applications.
	\begin{proof}[\bf Proof of Theorem \ref{main-thm-BC-p-1}]
		We argue by contradiction. Suppose that $K_Y$ is pseudo-effective. Let $\mF$ be the reflexive extension of $f^*K_Y$ and so it's pseudo-effective by Lemma \ref{lem-pseudo-effective}. given in Definition \ref{defn-psi}. Since $X$ is compact and $\psi$ is upper semi-continuous, there exists some $x\in X$ such that $\psi(x)$ achieves the maximum of $\psi$. Then $\psi(x)\neq 0$, otherwise $\{\eta_\alpha=-\infty\}\cup \{\varphi_\alpha=-\infty\}=U_\alpha$, which is a contradiction. By Proposition \ref{prop-subharmonic}, there exists $0\neq v\in T_{X,x}$ such that $\log\psi|_{L_v}$ is strictly subharmonic. Thus, there exists $r>0$ such that $\psi|_{L_v}$ has no maximum on $L_v\cap B(0,r)$. This contradict to that $\psi(x)$ is maximum.
	\end{proof}
	
	\begin{remark}
		From Remark \ref{remark-viewpoint}, Theorem \ref{main-thm-BC-p-1} is actually a special case of the following vanishing theorem: if $T_X$ is BC-$p$ positive for some $p \ge 1$, then $H^0(X,\Omega_X^p \otimes \mathcal{F}^{-1}) = 0$ for any pseudo-effective invertible sheaf $\mathcal{F}$. A proof of this statement might be obtained via Whitney’s comass, as suggested in \cite[Section 3]{ni2021}.
	\end{remark}
	
	\begin{proof}[\bf Proof of Theorem \ref{main-thm-BC-p-2} and Theorem \ref{main-thm-BC-RC}]
		Assume that $X$ is a compact K\"ahler manifold and its tangent bundle is BC-$p$ positive for every $p\geq k$. If its MRC fibration $f:X\dashrightarrow Y$ has generic fibre of dimension $< n-k+1$, then $\dim Y\geq k$.  Based on Lemma \ref{lem-MRCfibration}, $K_Y$ is pseudo-effective, which contradicts Theorem \ref{main-thm-BC-p-1}. Thus the proof of Theorem \ref{main-thm-BC-p-2} is complete. Combining Theorem \ref{main-thm-BC-p-2}, Theorem \ref{thm-RC-meancurvature} and Proposition \ref{prop-relationship}, we have the following relationship
		\[
		\begin{tikzcd}
			T_X\text{ is BC-$p$ positive for every }p\geq1 \arrow[d,Rightarrow] &  \\
			X\text{ is rationally connected} \arrow[r,Rightarrow] & T_X \text{ is mean curvature positive} \arrow[lu,Rightarrow]
		\end{tikzcd}
		\]
		which implies Theorem \ref{main-thm-BC-RC}.
	\end{proof}
	
    Theorem \ref{main-thm-orthogonal-positive} is an immediate consequence of Theorem \ref{main-thm-BC-p-2} and the following statement proved by L. Ni.
     \begin{proposition}(see e.g. \cite[Theorem 6.3]{ni2021})
     	Let $(X,g)$ be a compact K\"ahler manifold. If $\rc^\perp>0$, then $X$ is BC-$p$ positive for every $p\geq 1$.
     \end{proposition}
     
	 \noindent Theorem \ref{main-thm-BC-p-1} could be applicable to the broader curvature conditions considered in \cite{ni2021}, which ensure BC-$p$ positivity, such as $\rc_k^\perp,\sca_k^\perp$, etc. 
	 
	 \vspace{0.1cm}
	 
	 Now let us focus on proving Theorem \ref{main-thm-HSC-conformallyKahler}. It suffices to show the following proposition.
	 \begin{proposition}\label{prop-conformallykahler-BC}
	 	If $(X,h)$ is a compact conformally K\"ahler manifold with positive holomorphic sectional curvature, then its tangent bundle is BC-$p$ positive for every $p\geq1$.
	 \end{proposition}
	 
	 \begin{proof}
	 	Suppose that $X$ admits a conformally K\"ahler metric $h$ with positive holomorphic sectional curvature. Since $h$ is conformally K\"ahler, there exists some function $f\in C^\infty(X)$ such that $g=e^{-2f}h$ is K\"ahler. We shall verify that $(T_X,e^{(p+1)f}g)$ is $BC$-$p$ positive. A direct computation yields that
	 	$$0\leq R^h(v,\bar{v},u,\bar{u})=e^{2f}\big(R^g(v,\bar{v},u,\bar{u})-2g(v,\bar{v})\cdot\pp f(u,\bar{u})\big)$$
	 	for all $v,u\in T_{X,x}$, where $R^h,R^g$ are the Chern curvature tensors of $h,g$ of respectively.
	 	
	 	For any point $x\in X$, let $\{e_1,\cdots,e_p\}$ be a unitary frame of a $k$-dimensional subspace $\Sigma\subset T_{X,x}$ with respect to $e^{(p+1)f}g$, then $\{s_1:=e^{\frac{p+1}{2}f}e_1,\cdots, s_p=e^{\frac{p+1}{2}f}e_p\}$ is a unitary frame of $\Sigma$ with respect to $g$. Since $g$ is K\"ahler, it follows from Berger's average trick \cite{be55} (see e.g. \cite[page 8]{Li2022}) that
	 	\begin{align*}
	 		0&\leq e^{2f}\fint_{v\in \mathbb{S}^{2p-1}\subset \Sigma}\big(R^g(v,\bar{v},v,\bar{v})-2g(v,\bar{v})\cdot\pp f(v,\bar{v})\big)d\mu(v)\\
	 		&=e^{2f}\left(\frac{2}{p(p+1)}\sum\limits_{i,j=1}^pR^g(s_i,\overline{s_i},s_j,\overline{s_j})-\frac{2}{p}\sum\limits_{i=1}^p\pp f(s_i,\bar{s}_i)\right).
	 	\end{align*}
	 	Thus
	 	\begin{align*}
	 		&\sum\limits_{i,j=1}^pR^{e^{(p+1)f}g}(e_i,\overline{e_i},e_j,\overline{e_j})\\
	 		&=e^{-2(p+1)f}\sum\limits_{i,j=1}^p R^{e^{(p+1)f}g}(s_i,\overline{s_i},s_j,\overline{s_j})\\
	 		&=e^{-(p+1)f}\big(\sum\limits_{i,j=1}^pR^g(s_i,\overline{s_i},s_j,\overline{s_j})-(p+1)\sum\limits_{i=1}^p\pp f(s_i,\overline{s_i})\big)\geq0,
	 	\end{align*}
	 	which implies that there exists some $1\leq k\leq p$ such that $\sum\limits_{j=1}^pR^{e^{(p+1)f}g}(e_k,\overline{e_k},e_i,\overline{e_i})>0$ and the proof completes.
	 \end{proof}
	
	\section{Structure theorems for immediate curvature conditions}\label{section-structure}
	This section is devoted to establishing the structure theorems for $\sca_{a,b,k}\geq0$ (Theorem \ref{main-thm-partially}), which will be applied to conclude Theorem \ref{main-thm-kricci} and the quasi-positive case of Corollary \ref{main-thm-orthogonal-positive}. We introduce {\em the mixed partially curvature} $\sca_{a,b,k}$ for some $a,b\in\R$ and $k\in\N^+$ on a compact K\"ahler manifold $(X,g)$ as follows: for any $k$-dimensional subspace $\Sigma\subset T_xX$,
	$$\sca_{a,b,k}(x,\Sigma):=\fint_{v\in\Sigma,|v|_g=1}\big(a\rc(v,\bar{v})+b\hsc(v)\big)d\mu(v),$$
	where $d\mu$ is the volume measure of the standard unit sphere and $\hsc$ denotes the holomorphic sectional curvature of $g$, i.e.,
	$$\hsc(v):=\frac{R(v,\bar{v},v,\bar{v})}{|v|^4_g},\ \ \forall v\in T_{X,x}.$$
	We say that $\sca_{a,b,k}\geq0$ if $\sca_{a,b,k}(x,\Sigma)\geq0$ for any $x\in X$ and any $k$-dimensional subspace $\Sigma\subset T_xX$. We have the following basic statement. 
	\begin{lemma}\label{lem-k-k+1}
		For a compact K\"ahler manifold $(X,g)$, the following hold.
		
		(1) Let ${e_1,\cdots,e_k}$ be a unitary frame of $\Sigma$, we have
		$$\sca_{a,b,k}(x,\Sigma)=\frac{a}{k}\sum\limits_{i=1}^kR_{i\bar{i}}+\frac{2b}{k(k+1)}\sum\limits_{i,j=1}^kR_{i\bar{i}j\bar{j}}$$
		where $R_{i\bar{i}}=\rc({e_i,\overline{e_i}})$ and $R_{i\bar{i}j\bar{j}}=R(e_i,\overline{e_i},e_j,\overline{e_j})$.
		
		(2) For any integers $p\geq k$, $\sca_{a,b,k}\geq0\Rightarrow \sca_{a,b,p}\geq0$. Similar consequence holds for (quasi)-positivity.
	\end{lemma}

    For reader's convenience, we include a complete proof because it seems that there is no a proof even for $S_k\geq0$ in existing literature.

	\begin{proof}
		(1) directly follows from
		$$R_k(x,\Sigma):=\fint_{v\in\Sigma,|v|_g=1}\rc(v,\bar{v})d\theta(v)=\frac{1}{k}\sum\limits_{i=1}^kR_{i\bar{i}}$$\ and $$H_k(x,\Sigma):=\fint_{v\in\Sigma,|v|_g=1}\hsc(v)d\theta(v)=\frac{2}{k(k+1)}\sum\limits_{i,j=1}^k R_{i\bar{i}j\bar{j}}$$ based on Berger's average trick (see e.g. \cite[page 8]{Li2022}). Now let us prove (2). It suffices to show that $\sca_{a,b,k-1}\geq0\Rightarrow \sca_{a,b,k}\geq 0$ for all $k$. For any $v\in \Sigma$ with $|v|_g^2=1$, consider the $k-1$ dimensional subspace $\Sigma_v:=\{u\in\Sigma:u\perp v=0\}$. Then we have
		$$R_{k-1}(x,\Sigma_v)=\frac{k}{k-1}R_k(x,\Sigma)-\frac{1}{k-1}\rc(v,\bar{v})$$
		and
		$$H_{k-1}(x,\Sigma_v)=\frac{2}{k(k-1)}\left(\frac{k(k+1)}{2}H_k(x,\Sigma)-2\rc_k(x,\Sigma)(v,\bar{v})+\hsc(v)\right).$$
		Here $\rc_k(x,\Sigma)$ denotes the $k$-Ricci curvature along $\Sigma$ introduced in \cite{nicpam}, and $\rc_k(x,\Sigma)(v,\bar{v})=\sum\limits_{i=1}^k R_{i\bar{i}v\bar{v}}$. By applying Berger's average trick again, we have
		$$\fint_{v\in\Sigma,|v|_g=1}\rc_k(x,\Sigma)(v,\bar{v})d\theta(v)=\frac{1}{k}\sum\limits_{i,j=1}^kR_{i\bar{i}j\bar{j}}=\frac{k+1}{2}H_k(x,\Sigma).$$
		The proof is completed by
		\begin{align*}
		   &\fint_{v\in\Sigma,|v|_g=1}\sca_{a,b,k-1}(x,\Sigma_v)d\theta(v)\\
		   &= a\fint_{v\in\Sigma,|v|_g=1} R_{k-1}(x,\Sigma_v)d\theta(v)+b\fint_{v\in\Sigma,|v|_g=1}H_{k-1}(x,\Sigma_v)d\theta(v)\\
		   &=a(\frac{k}{k-1}-\frac{1}{k-1})R_k(x,\Sigma)+b\frac{2}{k(k-1)}(\frac{k(k+1)}{2}-(k+1)+1)H_k(x,\Sigma)\\
		   &=aR_k(x,\Sigma)+bH_k(x,\Sigma)
		   =\sca_{a,b,k}(x,\Sigma).
		\end{align*}
	\end{proof}
	
	We shall establish the following structure theorem by investigating the MRC fibration.
	
	\begin{theorem}\label{main-thm-partially}
		Let $(X,g)$ be a compact K\"ahler manifold with $\sca_{a,b,k}\geq0$ for some $a\geq0$, $a+\frac{2}{k+1}b>0$ and $k\in\N^+$. Then one of the following situations occurs:
		\begin{itemize}
			\item $\rd(X)\geq n-k+1$. 
			\item There exists a locally constant fibration $f:X\rightarrow Y$ to a compact K\"ahler manifold $Y$ such that
			\begin{itemize}
				\item[(1)] $Y$ supports a Ricci-flat K\"ahler metric $g_Y$.
				\item[(2)] The fibre $F$ is rationally connected and supports a K\"ahler metric $g_F$ such that we have the following holomorphically and isometrically splitting for the universal cover 
				$$(X_{\mathrm{univ}},\mu^*g)\simeq (Y_{\mathrm{univ}},\pi^*g_Y)\times(F,g_F).$$
			\end{itemize}
		\end{itemize}
	\end{theorem}
	
	For clarifying our arguments, the proof of Theorem \ref{main-thm-partially} will be divided into three parts.
	
	\subsection{Integral inequality}\label{section-bochner-integral}
	  Suppose that $\mF$ is a pseudo-effective invertible subsheaf of $\Omega_X^p$ on a compact K\"ahler manifold $(X,g)$ of dimension $n$ for some $1\leq p\leq n$. We adopt the notations in Section \ref{Hermitian ratio} and denote the K\"ahler form of $g$ by $\w$. Consider the Hermitian ratio $\psi$ associated to $\mF$ given by Definition \ref{defn-psi}. Since $\p\psi =e^{\varphi_\alpha}(\p|\eta_\alpha|^2_{\Lambda^pg^*}+|\eta_\alpha|_{\Lambda^pg^*}^2\p\varphi_\alpha)$, part (3) of Proposition \ref{prop-convergent} implies that $\partial\psi$ is priori $L^2$-integrable. The core of the proof of Theorem \ref{main-thm-partially} is establishing the following precise integral inequality.
	  \begin{proposition}\label{prop-integral-inequality-1}
	  	Suppose that $a\geq0$ and $\frac{a}{p}+\frac{2}{p(p+1)}b\geq 0$. When $\mF$ coincides with the reflexive extension of $f^*K_Y$ for a dominant meromorphic map $f:X\dashrightarrow Y$ from $X$ to a complex manifold $Y$, we have
	  	\begin{equation}\label{equa-integral-inequality-1}
	  		\begin{split}
	  			-(\frac{a}{p}+\frac{2b}{p(p+1)}) \int_X|\p\psi|^2\cdot\w^n\geq \int_X\psi^2\sca_{a,b,p}(x,\Sigma)\cdot\w^n,
	  		\end{split}
	  	\end{equation}
	  	where $\Sigma=\spaned\{e_1^*,\cdots,e_p^*\}$ and $e_1,\cdots,e_p\in T_{X,x}^*$ are given as in Proposition \ref{prop-Bochner} (1).
	  \end{proposition}     
	  
	  Before proceeding to the proof, we show some preliminary facts. First we note the following basic observation, which also shows that $\psi^2 \sca_{a,b,p}(x,\Sigma)$ is bounded and measurable; consequently, the integral on the right‑hand side of \eqref{equa-integral-inequality-1} is well defined and finite. Apart from this lemma, the rest of the proof of Proposition \ref{prop-integral-inequality-1} holds for a general pseudo-effective subsheaf $\mF$ of $\Omega_X^p$.
	  
	  \begin{lemma}\label{lem-mixedcurvature-geometriccondition} When $\mF$ satisfies the assumption in Proposition \ref{prop-integral-inequality-1}, we have that
	  	 \begin{equation}
	  	 	\begin{split}
	  	 	   \psi^2 \sca_{a,b,p}(x,\Sigma)\cdot\frac{\w^n}{n!}
	  	 	   &=-e^{\varphi_\alpha}\{\im R^{\Lambda^pg^*}\eta_\alpha,\overline{\eta_\alpha}\}_{\Lambda^pg^*}\wedge\bigg((\frac{a}{p}+\frac{2b}{p(p+1)})\psi\frac{\w^{n-1}}{(n-1)!}\\
	  	 	   &\ \ \ \ \ \ \ \ \ \ \ \   \ \ \  \  \  \ \ \  -\frac{2b}{p(p+1)}(\im)^{p^2}e^{\varphi_\alpha}\eta_\alpha\wedge\overline{\eta_\alpha}\wedge\frac{\w^{n-p-1}}{(n-p-1)!}\bigg)
	  	 	\end{split}
	  	 \end{equation}
	  \end{lemma}
	  
	  \begin{proof}
	  	Building on Proposition \ref{prop-Bochner} (1), for any $x\in X$, we can write $\eta_\alpha$ as $t_{\alpha,x}e_1\wedge\cdots\wedge e_p$ for some unitary frame $\{e_1,\cdots,e_n\}$ of $T_{X,x}^*$ and $t_{\alpha,x}\in\C$. For any $\beta=\im\sum\limits_{i,j=1}^n\beta_{i\bar{j}}e_i\wedge \overline{e_j}\in \Lambda^{1,1}T_{X,x}^*$ at $x$, a direct computations yields
	  	\begin{equation}\label{equa-wedge}
	  		\begin{split}
	  			&\beta\wedge(\im)^{p^2}e^{\varphi_\alpha}\eta_\alpha\wedge\overline{\eta_\alpha}\wedge\frac{\w^{n-p-1}}{(n-p-1)!}\\
	  			&=\psi \left(\im\sum\limits_{i,j=1}^n\beta_{i\bar{j}}e_i\wedge \overline{e_j}\right)\wedge \left(\bigwedge_{j=1}^p \im e_j\wedge\overline{e_j}\right) \wedge \left(\sum\limits_{i_1<\cdots<i_{n-p-1}}\bigwedge_{k=1}^{n-p-1} \im e_{i_k}\wedge \overline{e_{i_k}}\right)\\
	  			&=\psi\sum\limits_{i=p+1}^n\beta_{i\bar{i}}\cdot \im e_1\wedge\overline{e_1}\wedge\cdots\wedge \im e_n\wedge \overline{e_n}\\
	  			&=\psi \sum\limits_{i=p+1}^n\beta_{i\bar{i}}\cdot\frac{\w^n}{n!}
	  		\end{split}
	  	\end{equation}
	  	since $$e^{\varphi_\alpha}\eta_\alpha\wedge \overline{\eta_\alpha}=e^{\varphi_\alpha}|t_{\alpha,x_0}|^2(\im)^{p^2} e_1\wedge\cdots\wedge {e_p}\wedge\overline{e_1\wedge\cdots\wedge {e_p}}=\psi \bigwedge_{j=1}^p \im e_j\wedge\overline{e_j}$$ by $|t_{\alpha,x}|^2=|\eta_\alpha|_{\Lambda^pg^*}^2$. Then
	  	\begin{equation}\label{equa-wedge-mixed}
	  		\begin{split}
	  			&(\frac{a}{p}+\frac{2b}{p(p+1)})\psi\beta\wedge\frac{\w^{n-1}}{(n-1)!} -\frac{2b}{p(p+1)}\beta\wedge(\im)^{p^2}e^{\varphi_\alpha}\eta_\alpha\wedge\overline{\eta_\alpha}\wedge\frac{\w^{n-p-1}}{(n-p-1)!}\\
	  			&=\psi\bigg((\frac{a}{p}+\frac{2b}{p(p+1)})\sum\limits_{i=1}^n\beta_{i\bar{i}}-\frac{2b}{p(p+1)}\sum\limits_{i=p+1}^n\beta_{i\bar{i}}\bigg).
	  		\end{split}
	  	\end{equation}
	    Recall \eqref{equa-Curvature} in Proposition \ref{prop-Bochner} that
	  	\begin{align}\label{equa-lemma-curvature}
	  		-e^{\varphi_\alpha}\langle \im R^{\Lambda^pg^*}(\cdot,\cdot)\eta_\alpha,\overline{\eta_\alpha}\rangle_{\Lambda^pg^*}=\psi \im \sum\limits_{i=1}^p R(\cdot,\cdot,e_i^*,\overline{e_i^*})
	  	\end{align}
	  	and Lemma \ref{lem-k-k+1} gives
	  	\begin{equation}\label{equa-lemma-mixedcurvature}
	  		\begin{split}
	  			\sca_{a,b,p}(x,\Sigma)\cdot\w^n
	  			=&\frac{a}{p}\sum\limits_{i=1}^p\rc(e_i^*,\overline{e_i^*})+\frac{2b}{p(p+1)}\sum\limits_{i,j=1}^pR(e_i^*,\overline{e_i^*},e_j^*,\overline{e_j^*})\\
	  			=&(\frac{a}{p}+\frac{2b}{p(p+1)})\sum\limits_{j=1}^p \rc(e_j^*,\overline{e_j^*})-\frac{2b}{p(p+1)}\sum\limits_{i=p+1}^n\sum\limits_{j=1}^pR(e_i^*,\overline{e_i^*},e_j^*,\overline{e_j^*})
	  		\end{split}
	  	\end{equation}
	  	The proof is completed by substituting $\beta = \im\sum_{j=1}^p R(\cdot,\cdot,e_j^*,\overline{e_j^*})$ into \eqref{equa-wedge-mixed} and invoking \eqref{equa-lemma-curvature} together with \eqref{equa-lemma-mixedcurvature}.
	  \end{proof}
	  
	  Since the locally defined form $e^{\varphi_\alpha}\eta_\alpha\wedge\overline{\eta_\alpha}$ is actually global and lies in $L^\infty(X,\Lambda^{p,p}T_X^*)$, the associated current $\partial\bar\partial(e^{\varphi_\alpha}\eta_\alpha\wedge\overline{\eta_\alpha})$ is globally well-defined. Its vanishing is crucial to the proof of Proposition \ref{prop-integral-inequality-1}.
	  \begin{lemma}\label{lem-current-vanishing}
	  	In the sense of currents,
	  	\begin{equation}
	  		\begin{split}\label{equa-vanish-current}
	  			&(\im)^{p^2+1}\pp(e^{\varphi_\alpha}\eta_\alpha\wedge\overline{\eta_\alpha})=0,\\
	  			(\im\pp e^{\varphi_\alpha}&-e^{\varphi_\alpha}\im\p\varphi_\alpha\wedge\bp\varphi_\alpha)\wedge(\im)^{p^2}\eta_\alpha\wedge\overline{\eta_\alpha}=0.
	  		\end{split}
	  	\end{equation}
	  \end{lemma}
	  
	  \begin{proof}[Proof of Lemma \ref{lem-current-vanishing}]
	  	A direct computation yields that
	  	\begin{align*}
	  		&(\im)^{p^2+1}\pp(e^{\varphi_\alpha}\eta_\alpha\wedge\overline{\eta_\alpha})\\
	  		&=(\im)^{(p+1)^2}e^{\varphi_\alpha}\bigg( \p \varphi_\alpha\wedge\eta_\alpha\wedge \overline{\p\eta_\alpha}+\p\eta_\alpha\wedge \overline{\p\eta_\alpha}+\p\eta_\alpha\wedge\overline{\p {\varphi_\alpha}\wedge\eta_\alpha}\bigg)\\
	  		&\ \ \ \ \ +\im\pp e^{\varphi_\alpha}\wedge(\im)^{p^2}\eta_\alpha\wedge\overline{\eta_\alpha}\\
	  		&=(\im)^{(p+1)^2}e^{\varphi_\alpha}\tau_\alpha\wedge \overline{\tau_\alpha}+(\im\pp e^{\varphi_\alpha}
	  		-e^{\varphi_\alpha}\im \p\varphi_\alpha\wedge\bp\varphi_\alpha)\wedge(\im)^{p^2}\eta_\alpha\wedge\overline{\eta_\alpha}\\
	  		&\geq0
	  	\end{align*}
	  	in the sense of currents, where $\tau_\alpha=\p\eta_\alpha+\p\varphi_\alpha\wedge\eta_\alpha$. The first equality is valid since $\eta_\alpha$ is smooth and the fact that $\p e^{\varphi_\alpha}=e^{\varphi_\alpha}\p\varphi_\alpha$ in $L^1(X,\Lambda^{1,0}T_X^*)$ by \ref{prop-convergent} (1). The positivity of $(\im)^{p^2+1}\pp(e^{\varphi_\alpha}\eta_\alpha\wedge\overline{\eta_\alpha})$ follows because both terms in the fourth line are positive currents. The first term $(\im)^{(p+1)^2}e^{\varphi_\alpha}\tau_\alpha\wedge\overline{\tau_\alpha}$, is positive because it belongs to $L^1(X,\Lambda^{p+1,0}T_X^*)$ by Proposition \ref{prop-convergent} (3) and is pointwise positive. The positivity of the second term is a direct consequence of Proposition \ref{prop-convergent} (4). Note that the positive measure $(\im)^{p^2+1}\pp(e^{\varphi_\alpha}\eta_\alpha\wedge\overline{\eta_\alpha})\wedge\w^{n-p-1}$ must be zero because
	  	$$\left\langle(\im)^{p^2+1}\pp(e^{\varphi_\alpha}\eta_\alpha\wedge\overline{\eta_\alpha}),\w^{n-p-1}\right\rangle=\left\langle (\im)^{p^2+1}e^{\varphi_\alpha}\eta_\alpha\wedge\overline{\eta_\alpha},\pp(\w^{n-p-1})\right\rangle=0,$$
	  	so $(\im)^{p^2+1}\pp(e^{\varphi_\alpha}\eta_\alpha\wedge\overline{\eta_\alpha})$ is also identically zero, which completes the proof because the each term in the fourth line of the above inequality is positive and then must be zero.
	  \end{proof}
	  
	  \begin{remark}
	  	As a by product of the proof of Lemma \ref{lem-current-vanishing},
	  	$\p\eta_\alpha+\p\varphi_{\alpha}\wedge\eta_\alpha=0
	  	$
	  	almost everywhere. This recovers \cite[Main Theorem]{Demailly02} and \cite[Theorem 2]{Wu21} from the viewpoint explained in Remark \ref{remark-viewpoint}, i.e., for any pseudo-effective invertible sheaf $\mF$, the holomorphic sections of $H^0(X,\Omega_X^p\otimes \mF^{-1})$ are parallel with respect to the singular covariant derivative $\p_h=\p+\p\varphi_\alpha\wedge\cdot $.
	  \end{remark}
	  Now let us complete the proof of Proposition \ref{prop-integral-inequality-1}.
      
      \begin{proof}[Proof of Proposition \ref{prop-integral-inequality-1}]
         Since $\psi$ are bounded, $\psi\im\pp\psi$ is a well-defined current. We perform an integration by parts as follows. Consider the approximation of $\varphi_\alpha$ by plurisubharmonic functions $\varphi_{\alpha,\epsilon}$ as in Section \ref{currents-definition} and a partition of unity $\{\theta_\alpha\}$ subordinate to $\{U_\alpha\}$. Setting $\psi_\epsilon:=e^{\varphi_{\alpha,\epsilon}}|\eta|_{\Lambda^pg^*}^2$, we have
         \begin{equation}\label{equa-integrationbyparts}
         	\begin{split}
         		-&\int_X|\p\psi|_\w^2\cdot\frac{\w^n}{n!}
         		=-\int_X \im\p\psi\wedge\bp\psi\wedge\frac{\w^{n-1}}{(n-1)!}\\
         		\geq &-\lim\limits_{\epsilon\rightarrow 0}\sum\limits_\alpha\int_X\im\p\psi_\epsilon\wedge\bp\psi_\epsilon\wedge (\theta_\alpha\frac{\w^{n-1}}{(n-1)!}) \\	
         		=&\sum\limits_\alpha\lim\limits_{\epsilon\rightarrow 0}\left(\left\langle \psi_\epsilon\im\pp\psi_\epsilon,\theta_\alpha\frac{\w^{n-1}}{(n-1)!} \right\rangle+\left\langle \psi_\epsilon\im\p\psi_\epsilon,-\p\theta_\alpha\wedge\frac{\w^{n-1}}{(n-1)!} \right\rangle\right)\\
         		=&\left\langle \psi\im\pp\psi,\sum\limits_\alpha\theta_\alpha\frac{\w^{n-1}}{(n-1)!} \right\rangle+n\left\langle \psi\im\p\psi,-\p(\sum\limits_\alpha\theta_\alpha)\wedge\frac{\w^{n-1}}{(n-1)!} \right\rangle\\
         		= &\left\langle  \psi\im\pp\psi,\frac{\w^{n-1}}{(n-1)!}\right\rangle
         	\end{split}
         \end{equation}
         where the inequality in the second line follows from parts (1) and (3) of Proposition \ref{prop-convergent}, while the equality in the fourth line follows from parts (1) and (2) of the same proposition. For simplicity, we set
         \begin{align}\label{equa-simpli-defn-current}
         	\gamma= e^{\varphi_\alpha}\im\{D'_H\eta_\alpha,\overline{D'_H\eta_\alpha}\}_{\Lambda^pg^*},\ T=\left(\im\pp e^{\varphi_\alpha}-e^{\varphi_\alpha}\im\p\varphi_\alpha\wedge\bp\varphi_\alpha\right)|\eta_\alpha|_{\Lambda^pg^*}^2.
         \end{align}
         Building on Proposition \ref{prop-Bochner} (2), we have
         \begin{equation}\label{equa-Boch-simply}
         	 \im\pp\psi=\gamma+T-e^{\varphi_\alpha}\{\im R^{\Lambda^pg^*}\eta_\alpha,\overline{\eta_\alpha}\}_{\Lambda^pg^*}.
         \end{equation}
         It follows from (2) and (4) of Proposition \ref{prop-convergent} that $\gamma$ is a positive $(1,1)$-form that belongs to $L^1(X,\Lambda^{1,1}T_X^*)$ and $T$ is a positive current. Since $\psi=e^{\varphi_\alpha}|\eta_\alpha|^2_{\Lambda^pg^*}$ is non-negative, bounded and Borel measurable, $\psi T$ is a well-defined positive current via integration and in particular $\langle \psi T,\w^{n-1}\rangle\geq0$.
         
         Combining \eqref{equa-integrationbyparts} and \eqref{equa-Boch-simply} yields that
      	\begin{equation}\label{equa-integral-inequality-2}
      		\begin{split}
      			-\int_X|\p\psi|_\w^2\cdot\frac{\w^n}{n!}
      			\geq &\int_X\left(\psi \gamma-\psi e^{\varphi_\alpha}\{\im R^{\Lambda^pg^*}\eta_\alpha,\overline{\eta_\alpha}\}_{\Lambda^pg^*}\right)\wedge\frac{\w^{n-1}}{(n-1)!}.
      		\end{split}
      	\end{equation}
     Applying the same argument as in \eqref{equa-integrationbyparts} yields
     \begin{equation}\label{equa-integrationbyparts-2}
     	\begin{split}
     		\left\langle\im\pp\psi,(\im)^{p^2}e^{\varphi_\alpha}\eta_\alpha\wedge\overline{\eta_\alpha}\wedge\w^{n-p-1}\right\rangle
     		=&\left\langle \psi(\im)^{p^2+1}\pp(e^{\varphi_\alpha}\eta_\alpha\wedge\overline{\eta_\alpha}),\w^{n-p-1}\right\rangle\\
     		=&0
     	\end{split}
     \end{equation}
     where the second equality follows from Lemma \ref{lem-current-vanishing}. The same lemma also implies
     \begin{equation}\label{equa-vanishing-2}
     	\left\langle T,  (\im)^{p^2}e^{\varphi_\alpha}\eta_\alpha\wedge\overline{\eta_\alpha}\wedge\w^{n-p-1}\right\rangle= \left\langle T\wedge  (\im)^{p^2}e^{\varphi_\alpha}\eta_\alpha\wedge\overline{\eta_\alpha},\w^{n-p-1}\right\rangle=0
     \end{equation}
     Combining \eqref{equa-Boch-simply}, \eqref{equa-integrationbyparts-2} and \eqref{equa-vanishing-2}, we obtain
      \begin{equation}\label{equa-integral-equality-3}
      	\begin{split}
      	   0
      	   =&\int_X\left(\gamma-e^{\varphi_\alpha}\{\im R^{\Lambda^pg^*}\eta_\alpha,\overline{\eta_\alpha}\}_{\Lambda^pg^*}\right)\wedge (\im)^{p^2}e^{\varphi_\alpha}\eta_\alpha\wedge\overline{\eta_\alpha}\wedge\frac{\w^{n-p-1}}{(n-p-1)!}.
      	\end{split}
      \end{equation}
	  Taking the difference between
	 	 $(\frac{a}{p}+\frac{2b}{p(p+1)})\times$\eqref{equa-integral-inequality-2} and $\frac{2b}{p(p+1)}\times$\eqref{equa-integral-equality-3}, then applying Lemma \ref{lem-mixedcurvature-geometriccondition}, we obtain:
	 	 \begin{equation}\label{equa-integral-inequality-4}
	 	 	\begin{split}
	 	 			&-(\frac{a}{p}+\frac{2b}{p(p+1)})\int_X|\p\psi|_g^2\cdot\frac{\w^n}{n!}\\
	 	 		&\geq \int_X\bigg(\big(\frac{a}{p}+\frac{2b}{p(p+1)}\big)\psi\gamma\wedge\frac{\w^{n-1}}{(n-1)!} -\frac{2b}{p(p+1)}\gamma\wedge(\im)^{p^2}e^{\varphi_\alpha}\eta_\alpha\wedge\overline{\eta_\alpha}\bigg)\wedge\frac{\w^{n-p-1}}{(n-p-1)!}\\
	 	 		&\ \ +\int_X\psi^2\sca_{a,b,p}(x,\Sigma)\cdot\frac{\w^n}{n!}.
	 	 	\end{split}
	 	 \end{equation}
	 	 To obtain \eqref{equa-integral-inequality-1}, it remains to prove the non-negativity of the integral in the second line of \eqref{equa-integral-inequality-4}. Recall that $\gamma$ is a positive $(1,1)$-form belonging to $L^1(X,\Lambda^{1,1}T_X^*)$. Since $\psi$ is bounded, the form $\psi\gamma$ also belongs to $L^1(X,\Lambda^{1,1}T_X^*)$. It therefore suffices to verify that the integrand is non-negative for almost every $x\in X$. We adopt the notation from the proof of Lemma \ref{lem-mixedcurvature-geometriccondition} and write $\gamma=\sum\limits_{i,j=1}^n\gamma_{i\bar{j}}e_i\wedge \overline{e_j}$ almost everywhere. Substituting $\beta=\gamma$ into \eqref{equa-wedge-mixed} gives
	    \begin{align*}
	    	&\left(\big(\frac{a}{p}+\frac{2b}{p(p+1)}\big)\psi\gamma\wedge\frac{\w^{n-1}}{(n-1)!} -\frac{2b}{p(p+1)}\gamma\wedge(\im)^{p^2}e^{\varphi_\alpha}\eta_\alpha\wedge\overline{\eta_\alpha}\right)\wedge\frac{\w^{n-p-1}}{(n-p-1)!}\\
	    	&=\psi\bigg(\big(\frac{a}{p}+\frac{2b}{p(p+1)}\big)\sum\limits_{i=1}^n\gamma_{i\bar{i}}-\frac{2b}{p(p+1)}\sum\limits_{i=p+1}^n\gamma_{i\bar{i}}\bigg)\frac{\w^n}{n!}\\
	    	&=\psi\left(\big(\frac{a}{p}+\frac{2b}{p(p+1)}\big)\sum\limits_{i=1}^p\gamma_{i\bar{i}}+\frac{a}{p}\sum\limits_{i=p+1}^n\gamma_{i\bar{i}}\right)\frac{\w^n}{n!}\\
	    	&\geq0,
	    \end{align*}
	    where the last inequality is justified by the nonnegativity of $\gamma_{i\bar{i}}$, which follows from its definition in \eqref{equa-simpli-defn-current}. This substitution is justified because all computations in the proof of Lemma \ref{lem-mixedcurvature-geometriccondition} are performed pointwise. The proof is complete.
	 \end{proof}

	As a direct consequence of Proposition \ref{prop-integral-inequality-1}, Theorem \ref{main-thm-orthogonal-positive} can be extended to the quasi-positive case:
	\begin{corollary}\label{coro-orthogonal-quasi-positive}
		 A compact K\"ahler manifold with quasi-positive $\rc^\perp$ is rationally connected.
	\end{corollary}
	
	\begin{proof}
		Suppose otherwise. Then there exists an MRC fibration $f:X\dashrightarrow Y$ to a compact K\"ahler manifold $Y$ with $\dim Y = p > 0$ and $K_Y$ pseudo-effective. Since $\rc^\perp$ is quasi-positive, Lemma \ref{lem-k-k+1} implies that the function $\sca_{1,-1,p}(x):=\inf_\Sigma\sca_{1,-1,p}(x,\Sigma)$ is quasi-positive. $\sca_{1,-1,p}$ is continuous because $\sca_{1,-1,p}(x,\Sigma)$ can be viewed as a continuous  function on $U\times \operatorname{Gr}(p,n)$ for any coordinate neighborhood $U$. Consequently, there is a constant $\hat{\kappa}>0$ such that $S_{1,-1,p}\geq \hat{\kappa}$ on some open set $W$. Applying Proposition \ref{prop-integral-inequality-1} to the reflexive extension $\mF$ of $f^*K_Y$ yields that
		$$0\geq-(\frac{1}{p}-\frac{2}{p(p+1)})\int_X|\p\psi|^2\cdot\w^n\geq\hat{\kappa}\int_{W}\psi^2\cdot\w^n>0,$$
		where the last inequality follows from that $\{\psi=0\}:=\bigcup_\alpha\big(\{\eta_\alpha=0\}\cup\{\varphi_\alpha=-\infty\}\big)$ is a set of Lebesgue measure zero. This gives a contradiction, completing the proof.
	\end{proof}
	
	\subsection{Splitting theorem of the tangent bundle}\label{section-splitting} In this subsection, we prove the following splitting theorem, which is key ingredient of the structure theorems. The core of the proof is constructing a parallel $(1,1)$-form, which is based on the Bochner-type inequality \eqref{equa-integral-inequality-1} and the second variation argument in \cite{ni2021}.
	\begin{theorem}\label{thm-splitting}
		Let $(X,g)$ be a compact K\"ahler manifold with $\sca_{a,b,k}\geq0$ where $a\geq0$, $a+\frac{2}{k+1}b>0$ and $k\geq1$. Assume that there exists a meromorphic map $f:X\dashrightarrow Y$ from $X$ to a compact K\"ahler manifold $Y$ with $K_Y$ pseudoeffective such that $p:=\dim Y\geq k$. Let $Z$ be the indeterminacy of $f$ and $X_o:=X\setminus Z$. Then $f$ is non-degenerate on $X_o$ and there exists an integrable subbundle $V$ of $T_X$ such that $V$ is invariant by the action of the holonomy group and $V=T_{X/Y}$ on $X_o$. In particular,
        \begin{itemize}
            \item [(a)] $T_X$ admits a holomorphic orthogonal splitting $V\oplus V^\perp$.
            \item [(b)] $V$ and $V^\perp$ are integrable, and therefore $R(u,\bar{u},v,\bar{v})=0$ for any $u\in V, v\in V^\perp$ by de Rham decomposition. Furthermore, we have
        		$
        		\rc(u,\bar{u})=0
        		$
        		for all $u\in V.$ 
        \end{itemize}
	\end{theorem}
	
	\begin{proof}   
        Let $\mF$ be the invertible subsheaf of $\Omega_X^p$ given by the reflexive extension of $f^*K_Y$ and $\psi$ be the associated Hermitian ratio of $\mF$ given by Definition \ref{defn-psi}. We adopt the notations in Section \ref{Hermitian ratio}. Note that $\sca_{a,b,p}\geq0$ by Lemma \ref{lem-k-k+1} and $a+\frac{2}{p+1}b>0$. Applying Proposition \ref{prop-integral-inequality-1} to $\mF$ yields
        $$-\int_X|\p\psi|_\w^2\cdot\omega^n\geq0.$$
        Then $\p\psi=0$ in $L^2(X,\Lambda^{1,0}T_X^*)$ and so $\psi= C$ for some constant $C\geq0$ almost everywhere. Note that $U_\alpha\cap(\{\varphi_\alpha=-\infty\}\cup \{\eta_\alpha=0\})$ has zero  measure. Thus $C>0$, we conclude that $|\eta_\alpha|_\w^2>0$ since $\varphi_\alpha$ is bounded from above and $\varphi_\alpha=\log C-\log|\eta_\alpha|^2$ is smooth. A direct consequence of $|\eta_\alpha|^2>0$ is that $df$ is non-degenerate on $X_o$ in virtue of \eqref{equa-form-localexpression}.  
        
        Proposition \ref{prop-Bochner} implies that for any $x_0\in X$, there exists a holomorphic coordinate $(z^1,\cdots,z^n)$ centered at $x_0$ such that $\eta_\alpha(x_0)=t_{\alpha,x_0}dz^1\wedge\cdots\wedge dz^p$ for some $t_{\alpha,x_0}\in C^*$ and $\w(x_0)=\im\sum\limits_{i=1}^n dz^i\wedge d\bar{z}^i$; moreover, for any $u\in T_{X,x_0}$, it holds that
        	\begin{equation}\label{equa-boch3}
        		0=\psi^{-1}\p_u\p_{\bar{u}}\psi=\p_u\p_{\bar{u}}\varphi_\alpha+\frac{1}{|\eta_\alpha|^2}|D'_u\eta_\alpha+\eta_\alpha\p_u\varphi_\alpha|^2+\sum_{j=1}^p R_{u\bar{u}j\bar{j}}.
        	\end{equation}
        since $\psi$ is a constant and $\varphi_\alpha$ is smooth. We shall derive all needed geometric information from \eqref{equa-boch3}.
        \begin{lemma}\label{lem-geometric-imformation}
        	\begin{equation}\label{equa-parallel}
        		\im\pp\varphi_\alpha=0,\ \ D'\eta_\alpha+\eta_\alpha\p\varphi_\alpha=0.
        	\end{equation}
        	and \begin{equation}\label{equa-ricci-vanish}
        		\sum\limits_{j=1}^pR_{u\bar{u}j\bar{j}}=0  \text{\ \ \ for all\ }u\in T_{X,x_0}.
        	\end{equation}
        	for all $u\in T_{X,x_0}$
        \end{lemma}
        
        \eqref{equa-ricci-vanish} can be deduced from \eqref{equa-boch3} and \eqref{equa-parallel}. It suffices to verify $\eqref{equa-parallel}$.
        
        \begin{proof}[Proof of Lemma \ref{lem-geometric-imformation}]
        	 The discussion will be divided into two cases. Let $\Sigma:=\spaned\{\frac{\p}{\p z^i}\}_{i=1}^p$.
        	{\em Case 1: $a>0$.}  By $\sca_{a,b,p}\geq0$ and Lemma \ref{lem-k-k+1}, we have
        	\begin{equation}\label{equa-case1}
        		\begin{split}
        			0\leq & \sca_{a,b,p}(x,\Sigma)=\frac{a}{p}\sum\limits_{i=p+1}^n\sum\limits_{j=1}^pR_{i\bar{i}j\bar{j}}+\left(\frac{a}{p}+\frac{2b}{p(p+1)}\right)\sum\limits_{i,j=1}^pR_{i\bar{i}j\bar{j}}\\
        			= &-\frac{a}{p}\sum\limits_{i=p+1}^n(\p_i\p_{\bar{i}}\varphi_\alpha+\frac{1}{|\eta_\alpha|_g^2}|D'_i\eta_\alpha+\eta_\alpha\p_i\varphi_\alpha|^2)\\
        			&-\left(\frac{a}{p}+\frac{2b}{p(p+1)}\right)\sum\limits_{i=1}^p(\p_i\p_{\bar{i}}\varphi_\alpha+\frac{1}{|\eta_\alpha|_g^2}|D'_i\eta_\alpha+\eta_\alpha\p_i\varphi_\alpha|^2)
        		\end{split}
        	\end{equation}
        	where the last equality follows from the linear combinations of \eqref{equa-boch3}. Combining \eqref{equa-case1} and $\im\pp\varphi_\alpha\geq0$ yields \eqref{equa-parallel}, which implies \eqref{equa-ricci-vanish}.
        	
        	{\em Case 2: $a=0$.} In this case, $\sca_{a,b,p}\geq0\Leftrightarrow \sca_p\geq0$. Applying \eqref{equa-boch3} to $u=\frac{\p}{\p z_i}$ for $i=1,\cdots,p$ gives
        	\begin{align*}
        		0\leq & \sca_p(x_0,\Sigma)=\sum\limits_{i,j=1}^pR_{i\bar{i}j\bar{j}}=-\sum\limits_{i=1}^p(\p_i\p_{\bar{i}}\varphi_\alpha+\frac{1}{|\eta_\alpha|_g^2}|D'_i\eta_\alpha+\eta_\alpha\p_i\varphi_\alpha|^2)\leq0.
        	\end{align*}
        	Since $\im\pp\varphi_\alpha\geq0$, it follows that
        	\begin{equation}\label{equa-flat-1}
        		\p_i\p_{\bar{i}}\varphi_\alpha=0 \text{ and }  D'_i\eta_\alpha+\eta_\alpha\p_i\eta_\alpha=0 \text{ for } i=1,\cdots,p.
        	\end{equation}
            Then $ \sca_p(x_0,\Sigma)=\sum\limits_{i,j=1}^pR_{i\bar{i}j\bar{j}}=0$ and in particular $\sca_p(x_0,\cdot)$ attain its minimum at $\Sigma=\{\frac{\p}{\p z^1},\cdots,\frac{\p}{\p z^p}\}$. Applying the variation argument from \cite[Proposition 4.2]{ni2021}, we conclude that $\sum\limits_{j=1}^pR_{i\bar{i}j\bar{j}}\geq0$ for $i=p+1,\cdots,n$. Substituting $u=\frac{\p}{\p z^i}$ into \eqref{equa-boch3} for $i=p+1,\cdots n$, we get 
        	\begin{equation}\label{equa-flat-2}
        		\p_i\p_{\bar{i}}\varphi_\alpha=0 \text{ and }  D'_i\eta_\alpha+\eta_\alpha\p_i\eta_\alpha=0 \text{ for } i=p+1,\cdots,n.
        	\end{equation}
        	Combining \eqref{equa-flat-1} and \eqref{equa-flat-2} yields \eqref{equa-parallel}, which completes the proof.
        \end{proof}
            
        	Since $\im\pp\varphi_\alpha=0$, $\varphi_\alpha$ can be written as $t_\alpha+\overline{t_\alpha}$ for some holomorphic function $t_\alpha\in \mO(U_\alpha)$. Then $e^{t_\alpha}\eta_\alpha$ generates $\mF|_{U_\alpha}$ and we have
        	\begin{equation}
        		D(e^{t_\alpha}\eta_\alpha)=e^{t_\alpha}(D'\eta_\alpha+\eta_\alpha \p\varphi_\alpha)=0.
        	\end{equation}
        	In the remaining proof, we argue as in \cite[Section 3]{ZZ25}. Consider a global real $(1,1)$-form
        	\begin{align*}
        		\beta=\sum\limits_{i,j=1}^n\langle\iota_{\frac{\p}{\p z^i}}(e^{t_\alpha}\eta_\alpha),\overline{\iota_\frac{\p}{\p z^j}(e^{t_\alpha}\eta_\alpha)}\rangle_g dz^i\wedge d\bar{z}^j.
        	\end{align*}
        	Since $D(e^{t_\alpha}\eta_\alpha)=0$, it can be easily seen that $\beta$ is parallel with respect to $D$. Then we can define a parallel and self-adjoint linear transformation $\mathrm{P}:T_X\rightarrow T_X$ by $\langle \mathrm{P}(\cdot),\cdot\rangle_g=\beta(\cdot,{\cdot})$. Therefore, we obtain the decomposition
        	$$T_X\cong V\oplus W\ \ \text{on}\ X,$$
        	where $V$ and $W$ are generated by the eigenvectors corresponding to zero and nonzero eigenvalues of $\mathrm{P}$. Since $\mathrm{P}$ is parallel, $V$ and $W$ are the parallel distributions with respect to $D$, i.e., subbundles that are invariant by the parallel displacement.  By virtue of the de Rham decomposition of K\"ahler manifolds, the universal cover $\pi:(X_{\mathrm{univ}},g_{\mathrm{univ}})\rightarrow(X,g)$ holomorphically and isometrically splits as $$(X_{\mathrm{univ}},g_{\mathrm{univ}})\simeq(X_1,g_1)\times(X_2,g_2)$$ such that $T_{X_1}=\pi^*V$ and $T_{X_2}=\pi^*W$. In particular, $V\oplus W$ is the holomorphic orthogonal decomposition of $T_X$.

        	Recall that $T_{X/Y}$ is defined by the kernel of $df$, thus $T_{X/Y,x_0}=\spaned\{\frac{\p}{\p z^i}\}_{i=p+1}^n$ and $T_{X/Y,x_0}^\perp=\spaned\{\frac{\p}{\p z^i}\}_{i=1}^p$ using the holomorphic coordinates in the proof of Proposition \eqref{prop-Bochner} (1). By a direct computation, the explicit expression of $\beta$ at a point $x_0$ is
        	\begin{equation}\label{equa-localexpression}
        		\beta=\im\frac{|e^{t_\alpha}\eta_\alpha|^2}{p}\sum\limits_{i=1}^pdz^i\wedge d\bar{z}^i.
        	\end{equation} 	
        	This implies that
                $$T_{X/Y}=V|_{X_o}\ \ \text{and}\ \ T_{X/Y}^\perp=W|_{X_o}.$$
                 Thus, $R(u,\bar{u},v,\bar{v})=0$ for all $u\in T_{X/Y},v\in T_{X/Y}^\perp$.  Combining with \eqref{equa-ricci-vanish}, we get
        	\begin{equation}
        		\rc(u,\bar{u})=\sum\limits_{j=1}^nR_{u\bar{u}j\bar{j}}=\sum\limits_{j=1}^pR_{u\bar{u}j\bar{j}}+\sum\limits_{j=p+1}^nR_{u\bar{u}j\bar{j}}=0,
        	\end{equation}
        	and the proof is completed.
	\end{proof}

   \subsection{Proof of Theorem \ref{main-thm-partially}} Upon building the splitting theorem (Theorem \ref{thm-splitting}), the remaining part follows Matsumura's strategy on semi-positive holomorphic sectional curvature (\cite{matsumura2022}). In sake of completeness, we provide a brief proof. Combining Lemma \ref{lem-morphism} and Theorem \ref{thm-splitting}, we arrive at the following result.
     \begin{theorem}\label{thm-goodfibration}
         Let $(X,g)$ be a K\"ahler manifold with $\sca_{a,b,k}\geq 0$ for some $a+\frac{2b}{k+1}$ and $\rd(X)\leq n-k$. We can choose a MRC fibration of $X$ that is holomorphic.
     \end{theorem}
     
     \begin{proof}[Proof of Theorem \ref{thm-goodfibration}]
         Let $f:X\dashrightarrow Y$ be a MRC fibration of $X$. By virtue of Theorem \ref{thm-splitting}, there exists an integrable subbundle $V\subset T_X$ such that $V|_{X_o}=T_{X/Y}$. As $f$ is a MRC fibration, then the general leaf of $V$ is compact and rationally connected. By applying Lemma \ref{lem-morphism}, there exists a smooth submersion $f':X\rightarrow Y'$ such that $V=T_{X/Y'}$. A general fibre of $f$ is a leaf of $T_{X/Y'}$. Therefore there exists a bimeromorphism $\pi:Y\rightarrow Y'$ such that
         \[
         \begin{tikzcd}
             X \arrow[d,"f",dashrightarrow] \arrow[dr,"f'"] & \\
             Y \arrow[r,"\pi",dashrightarrow] & Y',
         \end{tikzcd}
         \]
         which implies that $Y'$ is not uniruled and $f'$ is the desired MRC fibration.
     \end{proof}

     The proof of Theorem \ref{main-thm-partially} is completed by showing the following statement and then applying it to a holomorphic MRC fibration constructed by Theorem \ref{thm-goodfibration}.
    
    \begin{theorem}\label{thm-structure-1}
    	Let $(X,g)$ be a compact K\"ahler manifold with $\sca_{a,b,k}\geq0$ where $a+\frac{2b}{k+1}\geq0$. Assume that there exists a surjective holomorphic map $f:X\rightarrow Y$ to a compact K\"ahler manifold $Y$ with $K_Y$ pseudo-effective such that $\dim Y\geq k$. The following statement holds:
    	\begin{itemize}
    		\item[(1)] $f$ is locally constant and induces a holomorphic orthogonal splitting of the short exact sequence
    		$$0\rightarrow (T_{X/Y},g_S)\rightarrow (T_X,g)\xrightarrow{df_*} (f^*T_Y,g_Q) \rightarrow0,$$
    		i.e., the orthogonal completement $T_{X/Y}^\perp$ is a holomorphic vector bundle and the isomorphism $(df_*)^\vee:f^*T_Y\rightarrow T_{X\setminus Y}^\perp$ is analytic such that we have the holomorphic orthogonal decomposition
    		$$T_X=T_{X/Y}\oplus T_{X/Y}^\perp=T_{X/Y}\oplus (df_*)^\vee(f^*T_Y)\cong T_{X/Y}\oplus f^*T_Y.$$
    		
    		\item[(2)] $Y$ supports a Ricci-flat K\"ahler metric $g_Y$.

    		\item[(3)] Let $F$ be the fibre of $f$ and $F_{\mathrm{univ}}$ be its universal cover. Then there exists a K\"ahler metric $g_{F_{\mathrm{univ}}}$ such that the universal cover $(X_{\mathrm{univ}},g_{\mathrm{univ}})$ of $(X,g)$ splits holomorphically and isometrically as
    		$$(X_{\mathrm{univ}},\mu^*g)\simeq (Y_{\mathrm{univ}},\pi^*g_Y)\times (F_{\mathrm{univ}},g_{F_\mathrm{univ}}).$$
    	\end{itemize} 
    \end{theorem}

     \begin{proof}[Proof of Theorem \ref{thm-structure-1}]
         By combining Lemma \ref{lem-integrability} and Theorem \ref{thm-splitting}, $f$ is locally constant, and statement (1) follows directly. Let $u,v$ be two holomorphic vector field on some open subset $U$ of $Y$, which induces two holomorphic section of $H^0(f^{-1}(U),f^*T_Y)$. We define $g_Y$ by
         $g_Y(u,\bar{v})=g_Q(f^*u,\overline{f^*v}).$ To check the well-definedness, it suffices to prove that for any $y\in Y$, $g_Q(f^*u,\overline{f^*v})$ is constant on the fibre $f^{-1}(y)$, which follows from
         $$\p_w\bp_wg_Q(f^*u,\overline{f^*v})\leq \langle R^{g_Q}(w,\bar{w})(f^*u),\overline{f^*v}\rangle_{g_Q}=R(w,\bar{w},f^*u,\overline{f^*v})=0$$
         for any $w\in T_{X/Y}=T_{f^{-1}(y)}$, where the second equality follows from the statement (1) and the third equality follows from Theorem \ref{thm-splitting} (b). Note that
         $f^*g_Y=\frac{\beta}{C}$ from the previous pointwise computation, where $C=|\psi|_g^2$ is a constant. Since $f$ is non-degenerate and $\beta$ is parallel, we deduce that $g_Y$ is K\"ahler. On the other hand, it follows from Theorem \ref{thm-splitting} (b)
         $$\rc^{g_Y}(u,\bar{u})=\sum\limits_{i=1}^p \langle R^{g_Y}(u,\bar{u})\frac{\p}{\p w^i},\overline{\frac{\p}{\p w^i}}\rangle_{g_Y}=\sum\limits_{i=1}^p\langle R^{g_Q}(f^*u,\overline{f^*u})f^*\frac{\p}{\p\w^i},\overline{f^*\frac{\p}{\p w^i}} \rangle_{g_Q}=0,$$
         where the last equality follows from the Gauss-Codazzi equation. The remaining proof is standard and we refer the readers to the argument in \cite[pp. 772-773]{matsumura2022}.
     \end{proof}

\vspace{0.1cm}

    \subsection{Structure theorems}\label{section-application}
     As applications of Theorem \ref{main-thm-partially}, in this section, we prove Theorem \ref{main-thm-kricci} and provide a complement of Chu-Lee-Zhu's result \cite{CLZ25} on non-negative mixed curvature. Here we write $\sca_k^g$ for the $k$-scalar curvature of the metric $g$. The notations $\rc_k^g$, etc., are defined analogously.
    \begin{proof}[Proof of Theorem \ref{main-thm-kricci}]
      By applying Theorem \ref{main-thm-partially} to $\sca_{1,0,k}$ and $\sca_{0,1,k}$, it suffices to show that $g_Y$ and $g_F$ satisfy the certain curvature condition. 
      
      When $\rc^g\geq_k0$, $\rc^{g_F}\geq0$ follows directly from the splitting of the metric $\mu^*g=\pi^*g_Y\times g_F$ and $\dim Y\geq k$. 
      
      Regarding $\sca_k^g\geq0$, $\sca_k^{g_Y}\equiv 0$ follows from a direct computation using the facts that Berger's trick applies, $\sca_k^g\geq0$ and $\sca_{\dim Y}^{g_Y}\equiv0$. While $\hsc^{g_F}=\sca_1^{g_F}\geq0$ follows from the fact that $\sca_{k+1}^g\geq0$, $\sca_k^{g_Y}\equiv0$, the splitting of the metric $\mu^*g=\pi^*g_Y\times g_F$ and $\dim Y\geq k$.
    \end{proof}
    
    \subsubsection{Mixed curvature} The mixed curvature $\mC_{a,b}$ are introduced in \cite{CLT22} as
    $$\mC_{a,b}(u,\bar{u})=a\rc(u,\bar{u})+\frac{b}{|u|_g^2}R(u,\bar{u},u,\bar{u}),\ \forall u\in T_X$$
    for providing an unified treatment for curvature conditions $\rc^\perp$, $\rc^+$ and $\rc_k$, those curvature conditions are considered by Ni-Zheng \cite{ni2018,ni2021,nicpam}. Chu-Lee-Zhu (\cite{CLZ25}) proved that the universal cover of a compact K\"ahler manifold $(X,g)$ with $\mC_{a,b}\geq0$ for $a>0$ and $a+b>0$ splits as
    $$(\C^k,g_{\C^k})\times\prod_i (X_i,h_i)$$
    where each $(X_j,h_j)$ are compact, simply connected and uniruled projective manifolds. However, it's unclear whether each $X_i$ is rationally connected. By applying Theorem \ref{main-thm-partially} to $\mC_{a,b}=\sca_{a,b,1}$, we completes their structure theorems in the sense of rational connectedness (Theorem \ref{thm-mixedcurvature}). For the quasi-positive case and the semi-positive case with irreducible holonomy, please refer to \cite{Tang2024,ni2025} and the references therein.
    
    \begin{theorem}\label{thm-mixedcurvature}
        Let $(X,g)$ be a compact K\"ahler manifold with $\mC_{a,b}\geq0$ for $a\geq0$ and $a+b>0$, then there exists a locally constant fibration $f:X\rightarrow Y$ from $X$ to a compact K\"ahler manfiold $Y$ such that
		\begin{itemize}
			\item[(1)] $Y$ supports a K\"ahler metric $g_Y$ with $\hsc\equiv0$. In particular, $Y$ is a compact complex torus.
			\item[(2)] The fibre $F$ is rationally connected and supports a K\"ahler metric $g_F$ such that we have the following holomorphically and isometrically splitting for the universal cover $\mu:X_{\mathrm{univ}}\rightarrow X$
			$$(X_{\mathrm{univ}},\mu^*g)\simeq (\C^k,g_{\C^k})\times(F,g_F).$$
		\end{itemize}
    \end{theorem}

     \begin{proof}[Proof of Theorem \ref{thm-mixedcurvature}]
         Suppose that $X$ is not uniruled, then $K_X$ is pseudo-effective by \cite{Ou25}. Applying Theorem \ref{thm-structure-1} to the identity map $id:X\rightarrow X$ and $\mC_{a,b}=\sca_{a,b,1}$, we deduce that $\rc\equiv 0$, which implies that the scalar curvature $\sca\equiv0$. By the Berger's trick, we have that $\mC_{a,b}\equiv0$, and thus $\hsc\equiv0$ whenever $a=0$ or $a\neq0$. Thus $Y$ is compact complex torus. If $X$ is uniruled, by combining with Lemma \ref{lem-MRCfibration} and Theorem \ref{main-thm-partially}, it suffices to verify that the holomorphic sectional curvature of $g_Y$ is identically zero, which follows from the similar argument.
     \end{proof}
     
     We conclude this section with some remarks on $\rc^\perp\geq0$. Similarly to Theorem \ref{thm-mixedcurvature}, we present the following structure theorems.
     \begin{theorem}\label{thm-orthogonal-semipositive}
     	Let $(X,g)$ be a compact K\"ahler manifold with $\rc^\perp\geq0$, then one of the following situations occurs:
     	\begin{itemize}
     		\item $\rd(X)\geq n-1$.
     		\item There exists a locally constant fibration $f:X\rightarrow Y$ from $X$ to a compact K\"ahler manfiold $Y$ such that
     		\begin{itemize}
     			\item[(1)] $Y$ supports a K\"ahler metric $g_Y$ with $\hsc\equiv0$. In particular, $Y$ is a compact complex torus.
     			\item[(2)] The fibre $F$ is rationally connected and supports a K\"ahler metric $g_F$ such that we have the following holomorphically and isometrically splitting for the universal cover $\mu:X_{\mathrm{univ}}\rightarrow X$
     			$$(X_{\mathrm{univ}},\mu^*g)\simeq (\C^k,g_{\C^k})\times(F,g_F).$$
     		\end{itemize}
     	\end{itemize}
     \end{theorem}
     While we can conclude the rational connectedness in the quasi-positive case. Regarding $\rc^\perp\geq0$, the left hand term of the Bochner-type formula \eqref{equa-integral-inequality-1} vanishes in the case $\rd(X)=n-1$. Thus we cannot conclude that the hermitian ratio $\psi\equiv C$, and the proof of splitting theorem in Section \ref{section-splitting} fails. It seems that we need a new approach to obtain the complete structure theorems. Unfortunately, we haven't found such a way at the moment.

	\bigskip

	\bibliographystyle{plain}

\end{document}